\documentclass[a4paper,12pt]{article}

\newenvironment{proof}{\noindent {\bf Proof:}}{\hfill $\Box$}

\newtheorem{theorem}{Theorem}
\newtheorem{lemma}{Lemma}

\newtheorem{corollary}{Corollary}

\newtheorem{assumption}{Assumption}
\newtheorem{remark}{Remark}

\usepackage{booktabs}

\usepackage{amsmath}
\usepackage{amssymb}
\usepackage{amsfonts}
\usepackage{graphicx}

\title{\bf Convex computation of the region of attraction of polynomial control systems}

\begin{document}

\author{Didier Henrion$^{1,2,3}$, Milan Korda$^4$}

\footnotetext[1]{CNRS; LAAS; 7 avenue du colonel Roche, F-31400 Toulouse; France. {\tt henrion@laas.fr}}
\footnotetext[2]{Universit\'e de Toulouse; LAAS; F-31400 Toulouse; France.}
\footnotetext[3]{Faculty of Electrical Engineering, Czech Technical University in Prague,
Technick\'a 2, CZ-16626 Prague, Czech Republic.}
\footnotetext[4]{Laboratoire d'Automatique, \'Ecole Polytechnique F\'ed\'erale de Lausanne, Station 9,
CH-1015, Lausanne, Switzerland. {\tt milan.korda@epfl.ch}}

\date{Draft of \today}

\maketitle

\begin{abstract}
We address the long-standing problem of computing the region of attraction (ROA) of a target set
(e.g., a neighborhood of an equilibrium point) of a controlled nonlinear system with polynomial
dynamics and semialgebraic state and input constraints. We show that the ROA can be computed by solving an infinite-dimensional convex linear programming (LP) problem over the space of measures. In turn, this problem can be solved approximately via a classical converging hierarchy of convex finite-dimensional
linear matrix inequalities (LMIs). Our approach is genuinely primal in the sense that convexity
of the problem of computing the ROA is an outcome of optimizing directly over system trajectories.
The dual infinite-dimensional LP on nonnegative continuous functions (approximated by polynomial sum-of-squares) allows us to generate a hierarchy of semialgebraic outer approximations of the ROA
at the price of solving a sequence of LMI problems with asymptotically vanishing conservatism.
This sharply contrasts with the existing literature which follows an exclusively dual Lyapunov
approach yielding either nonconvex bilinear matrix inequalities or conservative LMI conditions. The approach is simple and readily applicable as the outer approximations are the outcome of a single semidefinite program with no additional data required besides the problem description.
\end{abstract}

\begin{center}\small
\end{center}

\section{Introduction}

Given a nonlinear control system, a state-constraint set and a target set
(e.g. a neighborhood of an attracting orbit or an equilibrium point),
the constrained controlled region of attraction (ROA) is the set of all initial states
that can be steered with an admissible control
to the target set without leaving the state-constraint set. The target set can be required to be reached at a given time or at any time before a given time. The problem of computing the ROA (and variations thereof) lies at the heart of viability theory~(see, e.g.,~\cite{aubinViable}) and goes by many other names, e.g., the reach-avoid or target-hitting problem~ (see, e.g., \cite{lygeros}); the ROA itself is sometimes referred to as the backward reachable set~\cite{mitchell} or capture basin~\cite{aubinViable}.

There are many variations on the ROA computation problem addressed in this paper as well as a large number of related problems. For instance, one could consider asymptotic convergence instead of finite-time formulation and / or inner approximations instead of outer approximations. Among the related problems we can name the computation of (forward) reachable sets and maximum / minimum (robust) controlled invariant sets. Most of these variations are amenable to our approach, sometimes with different quality of the results obtained; see the Conclusion for a more detailed discussion.

We show that the computation of the ROA boils
down to solving an infinite-dimensional linear programming (LP) problem in the
cone of nonnegative Borel measures. The formulation is genuinely primal in the sense
that we optimize directly over controlled trajectories modeled with occupation measures
\cite{sicon,quincampoix}.

In turn, in the case of polynomial dynamics, semialgebraic state-constraint,
input-constraint and target sets, this infinite-dimensional LP can be solved approximately by a classical hierarchy of finite-dimensional convex linear matrix inequality (LMI) relaxations.
The dual infinite-dimensional LP on nonnegative continuous functions and its LMI relaxations
on polynomial sum-of-squares provide explicitly an asymptotically converging
sequence of nested semialgebraic outer approximations of the ROA.

The benefits of our occupation measure approach are overall the convexity of the problem of finding the ROA, and the availability of publicly available software to implement and solve the hierarchy of LMI relaxations.

Most of the existing literature on ROA computation follows Zubov's approach
\cite{margolis,zubov,hahn} and uses a dual Lyapunov certificate; see
\cite{vanelli}, the survey \cite{genesio}, Section 3.4 in \cite{haddad},
and more recently \cite{topcu,chesi} and \cite{chesiPaper} and the references therein.
These approaches either enforce convexity with conservative LMI conditions
(whose conservatism is difficult if not impossible to evaluate systematically)
or they rely on nonconvex bilinear matrix inequalities (BMIs),
with all their inherent numerical difficulties. In contrast, we show in this
paper that the problem of computing the ROA has actually a convex infinite-dimensional
LP formulation, and that this LP can be solved with a hierarchy of convex finite-dimensional LMIs
with asymptotically vanishing conservatism.

We believe that our approach is closer
in spirit to set-oriented approaches \cite{dellnitz}, level-set and Hamilton-Jacobi
approaches \cite{kurzhanski,lygeros,mitchell} or transfer operator
approaches \cite{vaidya}, even though we do not discretize with respect to time and/or
space. In our approach, we model a measure with a finite number of its moments,
which can be interpreted as a frequency-domain discretization (by analogy with
Fourier coefficients which are moments with respect to the unit circle).

Another way to evaluate the contribution of our paper is to compare it with
the recent works \cite{volume,pmi} which deal with polynomial approximations
of semialgebraic sets. In these references, the sets
to be approximated are given a priori (as a polynomial sublevel set,
or as a feasibility region of a polynomial matrix inequality).
In contrast, in the current paper the set to be approximated (namely the ROA
of a nonlinear dynamical system) is not known in advance, and our contribution
can be understood as an application and extension of the techniques of references
\cite{volume,pmi} to sets defined implicitly by differential equations.

For the special case of linear systems, the range of computational tools and theoretical results is wider; see, e.g., \cite{gilbertTan, blanchini}. Nevertheless, even for this simple class of systems, the problem of ROA computation is notoriously hard, at least in a controlled setting where, after time-discretization, polyhedral projections are required~\cite{blanchini}.

The use of occupation measures and related concepts has a long history in the fields of Markov decision processes and stochastic control; see, e.g., \cite{hernandez-lerma,fleming}. Applications to deterministic control problems were, to the best of our knowledge, first systematically treated\footnote{J. E. Rubio in \cite{rubio} used the so called Young measures~\cite{young}, not the occupation measures, although the idea of ``linearizing'' the nonlinear problem by lifting it into an infinite-dimensional space of measures is the same. Similar techniques were studied around the same time by others, for instance, A. F. Filippov, J. Warga and R. V. Gamkrelidze.} in~\cite{rubio} and enjoyed a resurgence of interest in the last decade; see, e.g., \cite{rantzer,sicon,vaidyaLyapunov,quincampoix} and references therein. However, to the best of our knowledge, this is the first time these methods are applied to region of attraction computation.



Our primary focus in this paper is the computation of the constrained finite-time controlled region of attraction of a given set. This problem is formally stated in Section~\ref{sec:ProblemStatement} and solved using occupation measures in Section~\ref{sec:primal}; the occupation measures themselves are introduced in Section~\ref{sec:OccupMeas}. A dual problem on the space of continuous functions is discussed in Section~\ref{sec:dual}. The hierarchy of finite-dimensional LMI relaxations of the infinite dimensional LP is described in Section~\ref{sec:LMI}. Convergence results are presented in Section~\ref{sec:convResults}. An extension to the free final time case is described in Section~\ref{sec:infiniteTime}. Numerical examples are presented in Section~\ref{sec:NumeEx}, and we conclude in Section~\ref{sec:Conclusion}.

A reader interested only in the finite-dimensional relaxations providing the converging sequence of outer approximations can consult directly the dual infinite-dimensional LP~(\ref{vlp}), its finite-dimensional LMI approximations~(\ref{dlmi}) and the resulting outer approximations in Section~\ref{sec:convResults} with convergence proven in Theorem~\ref{thm:convSet}.

\paragraph{Notation and preliminaries}

Throughout the paper the symbol $\mathbb{Z}_{[i,j]}$ denotes the set of consecutive integers $\{i,i+1,\ldots,j\}$. The symbol $\mathbb{R}[\cdot]$ denotes the ring of polynomials in variables given by the argument. The difference of two sets $A$ and $B$ is denoted by $A\setminus B$. All sets are automatically assumed Borel measurable; in particular if we write ``for all sets'', we mean ``for all Borel measurable sets''. By a measure $\mu$ defined on a set $X\subset \mathbb{R}^n$, we understand a signed Borel measure, i.e., a (not necessarily nonnegative) mapping from subsets of $X$ to real numbers satisfying $\mu(\emptyset) = 0$ and $\mu(\cup_{i=1}^\infty A_i) = \sum_{i=1}^\infty \mu(A_i)$ for every countable pairwise disjoint collection of sets $A_i\subset X$. The space of (signed) measures on a set $X$ is denoted by $M(X)$ and the space of continuous functions is denoted by $C(X)$; the dual space of all linear functionals on $C(X)$ is denoted by $C(X)'$. The symbols $C(X;\mathbb{R}^m)$ (resp. $C^1(X;\mathbb{R}^m)$) then denote the spaces of all continuous (resp. continuously differentiable) functions taking values in $\mathbb{R}^m$. Any measure $\mu \in M(X)$ can be viewed as an element of $C(X)'$ via the duality pairing induced by integration of functions in $C(X)$ with respect to~$\mu$:
\[
\langle \mu,g\rangle := \int _X g(x)\,d\mu(x),\; \mu\in M(X),\,g\in C(X).
\]
In addition, we use the following notation:
\begin{itemize}
\item
$I_A(\cdot)$ denotes the indicator function of a set $A$, i.e., a function
equal to $1$ on $A$ and $0$ elsewhere;
\item
$\lambda$ denotes the Lebesgue measure on $X \subset {\mathbb R}^n$, i.e.,
\[
\lambda(A) = \int_X I_A(x)\,d\lambda(x) = \int_X I_A(x)dx = \int_A dx
\]
is the standard $n$-dimensional volume of a set $A \subset X$;
\item
$\mathrm{spt}\,\mu$ denotes the support of a measure $\mu$, that is, the set of all points $x$
such that $\mu(A) > 0$ for every open neighborhood $A$ of $x$; by construction this set is closed.
\end{itemize}

\section{Problem statement}\label{sec:ProblemStatement}

Consider the control system
\begin{equation}\label{sys}
\dot{x}(t) = f(t,x(t),u(t)), \: t \in [0,T],
\end{equation}
with the state $x(t)\in \mathbb{R}^n$, the control input $u(t)\in\mathbb{R}^m$ and a terminal time $T > 0$. Each entry of the vector field $f$ is assumed to be polynomial\footnote{Note that the infinite-dimensional results of Sections~\ref{sec:primal} and \ref{sec:dual} hold with any Lipschitz $f$; the assumption of $f$ being polynomial and the constraint sets semialgebraic is required only for finite-dimensional relaxations of Section~\ref{sec:LMI}.}, i.e., $f_i \in {\mathbb R}[t,x,u]$,
$i\in\mathbb{Z}_{[1,n]}$. The state and the control input are subject to the basic semialgebraic constraints
\begin{equation}\label{con}
\begin{array}{l}
x(t) \in X \!:=\hspace{-0.1em} \{x \in {\mathbb R}^n : g^{X}_{i}(x) \geq 0, i\in\mathbb{Z}_{[1,n_X]}\},\, t\in [0,T], \\ 
u(t) \in U :=\! \{u \in {\mathbb R}^m : {g^U_i}(u) \geq 0, i\in\mathbb{Z}_{[1,n_U]}\}, \, t\in [0,T],\\
\end{array}
\end{equation}
with ${g^X_i} \in {\mathbb R}[x]$ and ${g^U_i} \in {\mathbb R}[u]$; the terminal state $x(T)$ is constrained to lie in the basic semialgebraic set
\[
X_T := \{x \in {\mathbb R}^n \: :\: {g^{X_T}_i}(x) \geq 0, i\in\mathbb{Z}_{[1,n_T]}\} \subset X,
\]
with ${g^{X_T}_i} \in {\mathbb R}[x]$.
\begin{assumption}\label{as:compact}
The sets $X$, $U$ and $X_T$ are compact.
\end{assumption}

A measurable control function $u:[0,T]\to \mathbb{R}^m$ is called admissible if $u(t)\in U$ for all $t\in [0,T]$. The set of all admissible control functions is denoted by $\mathcal{U}$. The set of all admissible trajectories starting from the initial condition $x_0$ generated by admissible control functions is then
\begin{align}\label{eq:calX0}
\mathcal{X}(x_0) := \Big\{ x(\cdot) \: : \: \exists\: u(\cdot)\in\mathcal{U}\;\mathrm{s.t.}\; &\dot{x}(t) = f(t,x(t),u(t))\; \mathrm{a.e.},\nonumber\\ 
& x(0) = x_0,\,x(t)\in X,\, x(T)\in X_T,\, \forall\, t\in [0,T]\Big\}, 
\end{align}
where $x(\cdot)$ is required to be absolutely continuous. Here, a.e. stands for ``almost everywhere'' with respect to the Lebesgue measure on $[0,T]$.

The constrained controlled region of attraction (ROA) is then defined as
\begin{equation}\label{eq:ROA}
	X_0 := \big\{x_0\in X \: : \: \mathcal{X}(x_0)\neq \emptyset \big\}.
\end{equation}
In words, the ROA is the set of all initial conditions for which there exists an admissible trajectory, i.e., the set of all initial conditions that can be steered to the target set in an admissible way. The set $X_0$ is bounded (by Assumption~\ref{as:compact}) and unique.


In the sequel we pose the problem of computing the ROA as an infinite-dimensional linear program (LP) and show how the solution to this LP can be approximated with asymptotically vanishing conservatism by a sequence of solutions to linear matrix inequality (LMI) problems.

\section{Occupation measures}\label{sec:OccupMeas}
In this section we introduce the concept of occupation measures and discuss its connection to trajectories of the control system~(\ref{sys}).

\subsection{Liouville's equation}
Given an initial condition $x_0$ and an admissible trajectory $x(\cdot \!\mid\! x_0) \in {\mathcal X}(x_0)$
with its corresponding control $u(\cdot \!\mid\! x_0) \in \mathcal{U}$, which we assume
to be a measurable function of $x_0$, we define the \emph{occupation measure} $\mu(\cdot\!\mid\! x_0)$ by
\[
\mu(A\times B\times C \!\mid\! x_0) := \int_0^T I_{A\times B\times C}(t,x(t \!\mid\! x_0),u(t \!\mid\! x_0))\,dt
\]
for all $A\times B\times C\subset
[0,T]\times X\times U$. For any such triplet of sets, the quantity $\mu(A\times B\times C \!\mid\! x_0)$ is equal to the amount of time out of $A\subset [0,T]$ spent by the state and control trajectory $(x(\cdot \!\mid\! x_0),u(\cdot \!\mid\! x_0))$ in $ B \times C \subset X\times U$.


 The occupation measure has the following important property: for any measurable function $g(t,x,u)$ the equality
\begin{equation}\label{eq:occProp}
\int_0^T g(t,x(t\!\mid\! x_0), u(t\!\mid\! x_0))\,dt = \int_{[0,T]\times X\times U}\hspace{-1.1cm}g(t,x,u)\,d\mu(t,x,u \!\mid\! x_0)
\end{equation}
holds. Therefore, loosely speaking, the occupation measure $\mu(\cdot \!\mid\! x_0)$ encodes all information about the trajectory of the state and control input $(x(\cdot \!\mid\! x_0),u(\cdot \!\mid\! x_0))$.

Define further the linear
operator $\mathcal{L}:C^1([0,T]\times X)\to C([0,T]\times X\times U)$ by
\[
v \mapsto \mathcal{L}v := \frac{\partial v}{\partial t}
+ \sum_{i=1}^n \frac{\partial v}{\partial x_i} f_i(t,x,u) = \frac{\partial v}{\partial t} + \mathrm{grad}\, v \cdot f
\] 
and its adjoint operator  $\mathcal{L}':C([0,T]\times X\times U)'\to C^1([0,T]\times X)'$ by the adjoint relation
\[
\langle \mathcal{L}'\nu,v\rangle := \langle \nu,\mathcal{L}v \rangle = \int_{[0,T]\times X \times U}\hspace{-2em} \mathcal{L}v(t,x,u)\, d\nu(t,x,u)
\]
for all $\nu\in M([0,T]\times X \times U)=C([0,T]\times X \times U)'$ and $v\in C^1([0,T]\times X)$. Given a test function $v \in C^1([0,T]\times X)$, 
it follows from~(\ref{eq:occProp}) that
\begin{align}\label{eq:basic}
v(T,x(T\!\mid\! x_0)) & \: = \:  v(0,x_0) \:+\: \displaystyle\int_0^T \frac{d}{dt}v(t,x(t \!\mid\! x_0))\,dt \nonumber\\
& \: = \: v(0,x_0) \:+\: \displaystyle\int_0^T  \mathcal{L}v(t,x(t \!\mid\! x_0),u(t \!\mid\! x_0))\,dt \nonumber\\
& \: = \: v(0,x_0) \:+\: \displaystyle\int_{[0,T]\times X\times U}\hspace{-2.5em} \mathcal{L}v(t,x,u)\,d\mu(t,x,u\!\mid\! x_0)\nonumber \\
& \: = \: v(0,x_0) \:+\: \langle \mathcal{L}'\mu(\cdot\!\mid\! x_0),v \rangle,
\end{align}
where we have used the adjoint relation in the last equality.

\begin{remark}
The adjoint operator $\mathcal{L}'$ is sometimes expressed symbolically as
\[
\nu \mapsto \mathcal{L}'\nu = -\frac{\partial \nu}{\partial t}
- \sum_{i=1}^n \frac{\partial(f_i \nu)}{\partial x_i} = -\frac{\partial \nu}{\partial t} - \mathrm{div}\, f\nu,
\]
where the derivatives of measures are understood in the weak sense, or 
in the sense of distributions (i.e., via their action on suitable test functions), and the change of sign comes from the integration by parts formula. For more details the interested reader is referred to any textbook on functional analysis and partial differential equation, e.g., \cite{evans}. The concept of weak derivatives of measures is not essential for the remainder of the paper and only highlights the important connections of our approach to PDE literature.
\end{remark}

\looseness-1
Now consider that the initial state is not a single point but that its distribution in space is modeled by an \emph{initial measure}\footnote{The measure $\mu_0$ can be thought of as the probability distribution of $x_0$ although we do not require that its mass be normalized to one.} $\mu_0\in M(X)$, and that for each initial state $x_0\in \mathrm{spt}\,\mu_0$ there exists an admissible trajectory $x(\cdot\!\mid\! x_0)\in \mathcal{X}(x_0)$ with an admissible control function $u(\cdot \!\mid\! x_0)\in \mathcal{U}$. Then we can define the average occupation measure $\mu\in M([0,T]\times X \times U)$ by
\begin{equation}\label{eq:ocmeas}
	\mu(A\times B\times C) := \int_X \mu(A\times B\times C \!\mid\! x_0)\, d\mu_0(x_0),
\end{equation}
and the \emph{final measure} $\mu_T\in M(X_T)$ by
\begin{equation}\label{eq:finmeas}
\mu_T(B) := \int_X I_B(x(T\!\mid\! x_0))\, d\mu_0(x_0).
\end{equation}
The average occupation occupation measure $\mu$ measures the average time spent by the state and control trajectories in subsets of $X \times U$, where the averaging is over the distribution of the initial state given by the initial measure $\mu_0$; the final measure $\mu_T$ represents the distribution of the state at the final time $T$ after it has been transported along system trajectories from the initial distribution $\mu_0$.

It follows by integrating (\ref{eq:basic}) with respect to $\mu_0$ that 
\begin{align*}
\int_{X_T} v(T,x)\,d\mu_T(x)   \: = \:  \int_X v(0,x)\,d\mu_0(x) +  \int_{[0,T]\times X\times U}\hspace{-2em}\mathcal{L}v(t,x,u) \,d\mu(t,x,u),
\end{align*}
or more concisely
\begin{equation}\label{test}
\langle\mu_T,\: v(T,\cdot)\rangle = \langle\mu_0,\: v(0,\cdot)\rangle \:\: +\:\:  \langle\mu,\:\mathcal{L}v\rangle \:\:\:\:\:\: \forall\, v \in C^1([0,T]\times X),
\end{equation}
which is a linear equation linking the nonnegative measures $\mu_T$, $\mu_0$ and $\mu$. Denoting $\delta_t$ the Dirac measure at a point $t$ and $\otimes$ the product of measures, we can write \[\langle\mu_0,\: v(0,\cdot)\rangle = \langle\delta_0\otimes\mu_0,\: v\rangle \quad\text{and}\quad \langle\mu_T,\: v(T,\cdot)\rangle=\langle  \delta_T\otimes\mu_T,\: v\rangle.\] Then, Eq.~(\ref{test}) can be rewritten equivalently using the adjoint relation as
\[
\langle{\mathcal L}'\mu,\:v\rangle = \langle\delta_T\otimes\mu_T,v\rangle \:-\: \langle\delta_0\otimes\mu_0,\: v\rangle\quad \forall\,v\in C^1([0,T]\times X),
\]
and since this equation is required to hold for all test functions $v$, we obtain
the linear operator equation
\begin{equation}\label{eq:Liouville}
\delta_T\otimes\mu_T = \delta_0\otimes\mu_0 + {\mathcal L}'\mu.
\end{equation}
This equation is classical in fluid mechanics and statistical physics, where $\mathcal{L}'$ is usually written using distributional derivatives of measures as remarked above; then the equation is referred to as Liouville's partial differential equation.

Each family of admissible trajectories starting from a given initial distribution $\mu_0\in M(X)$ 
satisfies Liouville's equation~(\ref{eq:Liouville}). The converse may not hold in general although for the computation of the ROA the two formulations can be considered equivalent, at least from a practical viewpoint. Let us briefly elaborate more
on this subtle point now.

\subsection{Relaxed ROA}\label{sec:relaxed}

The control system $\dot{x}(t) = f(t,x(t),u(t))$, $u(t)\in U$, can be viewed as a differential inclusion
\begin{equation}\label{eq:incl1} 
	\dot{x}(t) \in f(t,x(t),U) := \{f(t,x(t),u)\: : \: u\in U  \}.
\end{equation}
It turns out that in general the measures satisfying the Liouville's equation~(\ref{eq:Liouville}) are not in a one-to-one correspondence with the trajectories of~(\ref{eq:incl1}) but rather with the trajectories of the \emph{convexified inclusion}
\begin{equation}\label{eq:incl2}
	\dot{x}(t)\in \mathrm{conv}\:f(t,x(t),U),
\end{equation}
where $\mathrm{conv}$ denotes the convex hull\footnote{Note that the set $\mathrm{conv}\:f(t,x(t),U)$ is closed for every $t$ since $f$ is continuous and $U$ compact; therefore there is no need for a closure in~(\ref{eq:incl2}).}. Indeed, we show in Lemma~\ref{lem:corr} in Appendix~A that any triplet of measures satisfying
Liouville's equation~(\ref{eq:Liouville}) is generated by a family of trajectories of the convexified inclusion~(\ref{eq:incl2}). Here, given a family\footnote{Each such family can be described by a measure on $C([0,T];\mathbb{R}^n)$ which is supported on the absolutely continuous solutions to~(\ref{eq:incl2}). Note that there may be more than one trajectory corresponding to a single initial condition since the inclusion~(\ref{eq:incl2}) may admit multiple solutions.} of admissible trajectories of the convexified inclusion~(\ref{eq:incl2}) starting from an initial distribution $\mu_0$, the occupation and final measures are defined in a complete analogy via (\ref{eq:ocmeas}) and (\ref{eq:finmeas}), but now there are only the time and space arguments in the occupation measure, not the control argument.


Let us denote the set of absolutely continuous admissible trajectories of~(\ref{eq:incl2}) by
\begin{align*}
\bar{\mathcal X}(x_0) := \{x(\cdot):\;&\dot{x}(t)\in \mathrm{conv}\:f(t,x(t),U)\:\text{a.e.},\: x(0)=x_0,\\ &x(T)\in X_T,\: x(t)\in X\: \forall\, t\in[0,T]\}.
\end{align*}
 Then we define the relaxed region of attraction as
\[
\bar{X}_0 := \big\{x_0\in X \: : \: \bar{\mathcal{X}}(x_0)\neq \emptyset \big\}.
\]
Clearly $X_0\subset \bar{X}_0$ and the inclusion can be strict; see Appendix~C for concrete examples.
However, by the Filippov-Wa$\dot{\mathrm z}$ewski relaxation Theorem \cite[Theorem~10.4.4, Corollary~10.4.5]{aubin}, the trajectories of the original inclusion~(\ref{eq:incl1}) are dense (w.r.t. the metric of
uniform convergence of absolutely continuous functions of time) in the set of trajectories of the convexified inclusion~(\ref{eq:incl2}). This implies that the relaxed region of attraction $\bar{X}_0$ corresponds to the region of attraction of the original system but with infinitesimally dilated constraint sets $X$ and $X_T$; see Appendix~B for more details. Therefore, we argue that there is little difference between the two ROAs from a practical point of view. Nevertheless, because of this subtle distinction we make the following standing assumption in the remaining part of the paper.

\begin{assumption}\label{relaxed}
Control system (\ref{sys}) is such that $\lambda(X_0)=\lambda(\bar{X}_0)$.
\end{assumption}

\looseness-1
In other words, the volume of the classical ROA $X_0$ is assumed to be equal to
the volume of the relaxed ROA $\bar{X}_0$. Obviously, this is satisfied if
$X_0=\bar{X}_0$, but otherwise these sets may differ by a set of zero Lebesgue measure. Any of the following conditions on control system (\ref{sys})
is \emph{sufficient} for Assumption \ref{relaxed}
to hold:
\begin{itemize}
\item $\dot{x}(t) \in f(t,x(t),U)$ with $f(t,x,U)$ convex for all $t,x$,
\item $\dot{x}(t)=f(t,x(t))+g(t,x(t))u(t)$, $u(t) \in U$ with $U$ convex,
\item uncontrolled dynamics $\dot{x}(t)=f(t,x(t))$,
\end{itemize}
as well as all controllability assumptions allowing the application
of the constrained Filippov-Wa$\dot{\mathrm z}$ewski Theorem; see, e.g., \cite[Corollary 3.2]{frankowska}
and the discussion around Assumption I in~\cite{quincampoix}.

\subsection{ROA via optimization}
The problem of computing ROA $X_0$ can be reformulated as follows:
\begin{equation}\label{qlp}
\begin{array}{rcll}
q^* & = & \sup & \lambda(\mathrm{spt}\:\mu_0) \\
&& \mathrm{s.t.} &  \delta_T\otimes\mu_T = \delta_0\otimes\mu_0 + {\mathcal L}'\mu \\
&&& \mu\geq 0,\: \mu_0\geq 0,\: \mu_T\geq 0\\
&&& \mathrm{spt}\:\mu \subset [0,T]\times X\times U \\
&&& \mathrm{spt}\:\mu_0 \subset X, \:\: \mathrm{spt}\:\mu_T \subset X_T,
\end{array}
\end{equation}
where are $f$, $X$, $X_T$, $U$ are given data and the supremum is over a vector of nonnegative
measures $(\mu,\mu_0,\mu_T) \in M([0,T]\times X\times U)\times M(X)\times M(X_T)$.
Problem (\ref{qlp}) is an infinite-dimensional optimization problem on the cone
of nonnegative measures.

The rationale behind problem~(\ref{qlp}) is as follows. The first constraint is the Liouville's equation~(\ref{eq:Liouville}) which, along with the nonnegativity constraints, ensures that any triplet of measures $(\mu_0,\mu,\mu_T)$ feasible in~(\ref{qlp}) corresponds to an initial, an occupation and a terminal measure generated by trajectories of the controlled ODE~(\ref{sys}) (or more precisely of the convexified differential inclusion~(\ref{eq:incl2})). The support constraint on the occupation measure $\mu$ ensures that these trajectories satisfy the state and control constraints; the support constraint on $\mu_T$ ensures that the trajectories end in the target set. Maximizing the volume of the support of the initial measure then yields an initial measure with the support equal to the ROA up to a set of zero volume\footnote{Even though the support of the initial measure attaining the maximum in~(\ref{qlp}) can differ from the ROA on the set of zero volume, the outer approximations obtained in Section~\ref{sec:convResults} are valid ``everywhere'', not ``almost everywhere''.} (in view of Assumption~\ref{relaxed}). This discussion is summarized in the following Lemma.

\begin{lemma}\label{lem:qlp}
	The optimal value of problem~(\ref{qlp}) is equal to the volume of the ROA $X_0$, that is, $q^* = \lambda(X_0)$.
\end{lemma}
\begin{proof} By definition of the ROA, for any initial condition $x_0\in X_0$ there is an admissible trajectory in $\mathcal{X}(x_0)$. Therefore for any initial measure $\mu_0$ with $\mathrm{spt}\,\mu_0 \subset X_0$ there exist an occupation measure $\mu$ and a final measure $\mu_T$ such that the constraints of problem~(\ref{qlp}) are satisfied. Thus, $q^*\ge \lambda(X_0) = \lambda(\bar{X}_0)$, where the equality follows from Assumption~\ref{relaxed}.

\looseness-1
Now we show that $q^* \le \lambda(X_0)=\lambda(\bar{X}_0)$. For contradiction, suppose that a triplet of measures $(\mu_0,\mu,\mu_T)$ is feasible in~(\ref{qlp}) and that $\lambda(\mathrm{spt}\,\mu_0 \setminus \bar{X}_0) > 0$. From Lemma~\ref{lem:corr} in Appendix~A there is a family of admissible trajectories of the inclusion~(\ref{eq:incl2}) starting from $\mu_0$ generating the $(t,x)$-marginal of the occupation measure $\mu$ and the final measure $\mu_T$. However, this is a contradiction since no trajectory starting from $\mathrm{spt}\,\mu_0 \setminus \bar{X}_0$ can be admissible. Thus,  $\lambda(\mathrm{spt}\,\mu_0 \setminus \bar{X}_0) = 0$ and so $\lambda(\mathrm{spt}\,\mu_0) \le \lambda(\bar{X}_0)$. Consequently, $q^* \le \lambda(\bar{X}_0)=\lambda(X_0)$.
\end{proof}

\section{Primal infinite-dimensional LP on measures}\label{sec:primal}
The key idea behind the presented approach consists in replacing the direct maximization of the support of the initial measure $\mu_0$ by the maximization of the integral below the density of $\mu_0$ (w.r.t. the Lebesgue measure) subject to the constraint that the density be below one. This procedure is equivalent to maximizing the mass\footnote{The mass of the measure $\mu_0$ is defined as $\int_X1\, d\mu_0 = \mu_0(X)$.} of $\mu_0$ under the constraint that $\mu_0$ is dominated by the Lebesgue measure. 
This leads to the following infinite-dimensional LP:
\begin{equation}\label{rlp}
\begin{array}{rcll}
p^* & = & \sup & \mu_0(X) \\
&& \mathrm{s.t.} & \delta_T\otimes\mu_T = \delta_0\otimes\mu_0 + {\mathcal L}'\mu \\
&&& \mu\geq 0,\: \lambda\geq\mu_0\geq 0, \:\mu_T\geq 0\\
&&& \mathrm{spt}\:\mu \subset [0,T]\times X\times U\\
&&& \mathrm{spt}\:\mu_0 \subset X, \:\:\mathrm{spt}\:\mu_T \subset X_T,
\end{array}
\end{equation}
where the supremum is over a vector of nonnegative
measures $(\mu,\mu_0,\mu_T) \in M([0,T]\times X\times U)\times M(X)\times M(X_T)$.
In problem (\ref{rlp}) the constraint $\lambda\geq\mu_0$ means that $\lambda(A)\geq\mu_0(A)$
for all sets $A \subset X$. Note how the objective functions differ in problems (\ref{qlp})
and (\ref{rlp}).

The following theorem is then almost immediate.

\begin{theorem}\label{thm:1}
The optimal value of the infinite-dimensional LP problem (\ref{rlp}) is equal to the volume of the ROA $X_0$, that is, $p^*=\lambda(X_0)$.
Moreover, the supremum is attained with the $\mu_0$-component of the optimal solution equal to the restriction of the Lebesgue measure to the ROA $X_0$.
\end{theorem}
\begin{proof} Since the constraint set of problem (\ref{rlp}) is tighter than that of problem (\ref{qlp}), by Lemma~\ref{lem:qlp} we have that $\lambda(\mathrm{spt}\,\mu_0) \le \lambda(X_0)$ for any feasible $\mu_0$. From the constraint $\mu_0\le \lambda$ we get $\mu_0(X) = \mu_0(\mathrm{spt}\,\mu_0)\le \lambda(\mathrm{spt}\,\mu_0)\le\lambda(X_0)$ for any feasible $\mu_0$. Therefore $p^* \le \lambda(X_0)$. But by definition of the ROA $X_0$, the restriction of the Lebesgue measure to $X_0$ is feasible in~(\ref{rlp}), and so $p^* \ge \lambda(X_0)$. Consequently $p^* = \lambda(X_0)$.
\end{proof}

Now we reformulate problem~(\ref{rlp}) to an equivalent form more convenient for dualization and subsequent theoretical analysis. To this end, let us define the complementary measure (a slack variable) $\hat{\mu}_0\in M(X)$ such that the inequality $\lambda \ge \mu_0 \ge 0$ in~(\ref{rlp}) can be written equivalently
as the constraints $\mu_0+\hat{\mu}_0 = \lambda$, $\mu_0\ge0$, $\hat{\mu}_0\ge0$. Then problem~(\ref{rlp}) is equivalent to the infinite-dimensional primal LP
\begin{equation}\label{rrlp}
\begin{array}{rcll}
p^* & = & \sup & \mu_0(X) \\
&& \mathrm{s.t.} & \delta_T\otimes\mu_T = \delta_0\otimes\mu_0 + {\mathcal L}'\mu \\
&&& \mu_0 + \hat{\mu}_0 = \lambda\\
&&& \mu\geq 0, \: \mu_0\geq 0,\: \mu_T\geq 0,\: \hat{\mu}_0\ge0\\
&&& \mathrm{spt}\:\mu \subset [0,T]\times X\times U\\
&&& \mathrm{spt}\:\mu_0 \subset X, \:\:\mathrm{spt}\:\mu_T \subset X_T\\
&&& \mathrm{spt}\: \hat{\mu}_0 \subset X.
\end{array}
\end{equation}

\section{Dual infinite-dimensional LP on functions}\label{sec:dual}
In this section we derive a linear program dual to (\ref{rrlp}) (and hence to~(\ref{rlp})) on the space of continuous functions. A certain super-level set of one of the functions feasible in the dual LP will provide an outer approximation to the ROA $X_0$.

Consider the infinite-dimensional LP problem
\begin{equation}\label{vlp}
\begin{array}{rclll}
d^* & \hspace{-0.7em}= & \hspace{-0.5em} \inf  &  \hspace{-0.5em}\displaystyle\int_{X} w(x)\, d\lambda(x) \\
&& \hspace{-0.5em}\mathrm{s.t.} &  \hspace{-0.5em}\mathcal{L}v(t,x,u) \leq 0, &  \hspace{-1em}\forall\, (t,x,u) \in [0,T]\times X\times U \\
&&&  \hspace{-0.5em}w(x) \ge v(0,x) + 1, \:\: & \hspace{-1em}\forall\, x \in X \\
&&&  \hspace{-0.5em}v(T,x) \geq 0, \:\: & \hspace{-1em}\forall\, x \in X_T\\
&&&  \hspace{-0.5em}w(x) \geq 0, \:\: & \hspace{-1em}\forall\, x \in X,
\end{array}
\end{equation}
where the infimum is over $(v,w) \in C^1([0,T]\times X)\times C(X)$. The dual has the following interpretation: the constraint $\mathcal{L}v\le 0$ forces $v$ to decrease along trajectories and hence necessarily $v(0,x) \ge 0$ on $X_0$ because of the constraint $v(T,x)\ge 0$ on $X_T$. Consequently, $w(x) \ge 1$ on $X_0$. This instrumental observation is formalized in the following Lemma.

\begin{lemma}\label{lem:v0}
Let $(v,w)$ be a pair of function feasible in~(\ref{vlp}). Then $v(0,\cdot) \ge 0$ on $X_0$ and $w \geq 1$ on $X_0$.
\end{lemma}
\begin{proof}
By definition of $X_0$, given any $x_0\in X_0$ there exists $u(t)$ such that $x(t)\in X$, $u(t)\in U$ for all $t\in [0, T]$ and $x(T)\in X_T$. Therefore, since $v(T,\cdot)\ge 0$ on $X_T$ and $\mathcal{L}v\le 0$ on $[0,T]\times X\times U$,
\begin{align*}
  0 \le v(T,x(T)) &=v(0,x_0) + \int_0^T \frac{d}{dt}v(t,x(t))\,d t\\ &= v(0,x_0) + \int_0^T \mathcal{L}v(t,x(t),u(t))\,d t\\& \le v(0,x_0)\le w(x_0) -1,
\end{align*}
where the last inequality follows from the second constraint of~(\ref{vlp}).
\end{proof}

Next, we have the following salient result:
\begin{theorem}\label{thm:noGap}
There is no duality gap between the primal infinite-dimensional LP problems (\ref{rlp}) and (\ref{rrlp}) and the infinite-dimensional dual LP problem (\ref{vlp}) in the sense that
$p^*=d^*$.
\end{theorem}
\begin{proof}
To streamline the exposition, let
\begin{align*}
{\mathcal C} &:=C([0,T]\times X\times U)\times C(X)\times C(X_T)\times C(X),\\
{\mathcal M} &:=M([0,T]\times X\times U)\times M(X)\times M(X_T)\times M(X),
\end{align*}
and let $\mathcal{K}$ and $\mathcal{K}'$ denote the positive cones of $\mathcal{C}$ and $\mathcal{M}$ respectively. Note that the cone $\mathcal{K}'$ of nonnegative measures of $\mathcal M$ can be identified with the topological dual of the cone $\mathcal{K}$ of nonnegative continuous functions of $\mathcal C$. The cone $\mathcal{K}'$ is equipped with the weak-* topology; see \cite[Section 5.10]{luenberger}. Then, the LP problem (\ref{rrlp}) can be rewritten as
\begin{equation}\label{plp}
\begin{array}{rcll}
p^* & = & \sup & \langle\gamma,\:c\rangle \\
&& \mathrm{s.t.} & {\mathcal A}'\gamma = \beta \\
&&& \gamma \in \mathcal{K'},
\end{array}
\end{equation}
where the infimum is over the vector $\gamma:=(\mu,\mu_0,\mu_T,\hat{\mu}_0)$,
the linear operator ${\mathcal A}' : \mathcal{K'} \to C^1([0,T]\times X)' \times M(X)$
is defined by 
$
{\mathcal A}'\gamma := \left(-{\mathcal L}'\mu - \delta_0\otimes\mu_0 + \delta_T\otimes\mu_T\:\:,\:\: \mu_0+\hat{\mu}_0 \right),
$
the right hand side of the equality constraint in~(\ref{plp}) is the vector of measures
$ \beta := (0\,,\, \lambda) \in M([0,T]\times X)\times M(X),$
the vector function in the objective is $c := (0,1,0,0) \in {\mathcal C}$,
so the objective function itself is
\[
	\langle\gamma,\:c\rangle = \int_X d\mu_0 = \mu_0(X).
\]
The LP problem (\ref{plp}) can be interpreted
as a dual to the LP problem
\begin{equation}\label{dlp}
\begin{array}{rcll}
d^* & = & \inf & \langle \beta,\:z\rangle \\
&& \mathrm{s.t.} & {\mathcal A}(z)-c \in \mathcal{K},
\end{array}
\end{equation}
where the infimum is over $z:=(v,w) \in C^1([0,T]\times X)\times C(X)$, and the linear operator ${\mathcal A} : C^1([0,T]\times X)\times C(X) \to {\mathcal C}$
is defined by
\[
{\mathcal A}z := (-\mathcal{L}v,\:\: w-v(0,\cdot),\:\:v(T,\cdot),\:\: w)
\]
and satisfies the adjoint relation $\langle{\mathcal A}'\gamma,\:z\rangle \:=\: \langle\gamma,\:{\mathcal A}z\rangle$. The LP problem (\ref{dlp})  is exactly the LP problem (\ref{vlp}).

To conclude the proof we use an argument similar to that of \cite[Section C.4]{lasserre}.
From \cite[Theorem 3.10]{anderson} there is no duality gap between LPs (\ref{plp})
and (\ref{dlp}) if the supremum $p^*$ is finite and 
the set $P:=\{({\mathcal A}'\gamma,\:\langle \gamma,\:c\rangle) : \gamma \in \mathcal{K'}\}$
is closed in the weak-* topology of $\mathcal{K'}$. The fact that $p^*$ is finite
follows readily from the constraint $\mu_0 + \hat{\mu}_0 = \lambda$, $\hat{\mu}_0\ge 0$, and from compactness of $X$.
To prove closedness, we first remark that ${\mathcal A}'$ is weakly-* continuous\footnote{The weak-* topology on $C^1([0,T]\times X)'\times M(X)$ is induced by the standard topologies on $C^1$ and $C$ -- the topology of uniform convergence of the function and its derivative on $C^1$ and the topology of uniform convergence on $C$.} since ${\mathcal A}(z) \in {\mathcal C}$
for all $z \in C^1([0,T]\times X)\times C(X)$. Then we consider
a sequence $\gamma_k = (\mu^k,\mu_0^k,\mu_T^k,\hat{\mu}_0^k) \in \mathcal{K'}$ and we want to show that its
accumulation point $(\nu, a) := \lim_{k\to\infty}({\mathcal A}'\gamma_k,\:\langle \gamma_k,\:c\rangle)$
belongs to $P$, where $\nu \in C^1([0,T]\times X)' \times M(X)$ and $a \in \mathbb{R}$. To this end, consider first the test function $z_1 = (T-t, 1)$ which gives $\langle \mathcal{A}'\gamma_k,z_1\rangle = \mu^k(0,T\times X\times U) + \mu_0^k(X) + \hat{\mu}_0^k(X) \to \langle \nu, z_1\rangle < \infty$; since the measures are nonnegative, this implies that the sequences of measures $\mu^k$, $\mu_0^k$ and $\hat{\mu}_0^k$ are bounded. Next, taking the test function $z_2 = (1,1)$ gives $\langle \mathcal{A}'\gamma_k,z_2\rangle =  \mu_T(X) + \hat{\mu}_0(X) \to \langle \nu,z_2 \rangle < \infty $; this implies that the sequence $\mu_T^k$ is bounded as well. Thus, from the weak-* compactness of the unit ball (Alaoglu's Theorem \cite[Section 5.10, Theorem~1]{luenberger})
there is a subsequence $\gamma_{k_i}$ that converges \mbox{weakly-*}
to an element $\gamma \in \mathcal{K'}$ so that
$\lim_{i\to\infty}({\mathcal A}'\gamma_{k_i},\:\langle \gamma_{k_i},\:c\rangle) \in P$ by continuity of $\mathcal{A}'$.
\end{proof}


Note that, by Theorem~\ref{thm:1}, the supremum in the primal LPs~(\ref{rlp}) and (\ref{rrlp}) is attained (by the restriction of the Lebesgue measure to $X_0$). In contrast, the infimum in the dual LP (\ref{vlp}) is not attained in $C^1([0,T]\times X)\times C(X)$, but there exists a sequence of feasible solutions to~(\ref{vlp}) whose $w$-component converges to the discontinuous indicator $I_{X_0}$ as we show next.

Before we state our convergence results, we recall the following types of convergence of a sequence of functions $w_k:X\to \mathbb{R}$ to a function $w:X\to \mathbb{R}$ on a compact set $X \subset {\mathbb R}^n$.
As $k\to\infty$, the functions $w_k$ converge to $w$:
\begin{itemize}
\item in $L^1$ norm if $\int_X |w_k-w|\,d\lambda \to 0$,
\item in Lebesgue measure if $\lambda(\{x\: :\: |w_k(x)-w(x)|\ge \epsilon  \}) \to 0\ \forall\,\epsilon>0$,
\item almost everywhere if $\exists\, B\subset X$, $\lambda(B)=0$, such that $w_k \to w$ pointwise on $X\backslash B$,
\item almost uniformly if $\forall\, \epsilon > 0, \exists\, B\subset X$, $\lambda(B) < \epsilon$, such that $w_k\to w$ uniformly on $X\backslash B$. 
\end{itemize}
We also recall that convergence in $L^1$ norm implies convergence in Lebesgue measure and that almost uniform convergence implies convergence almost everywhere (see \cite[Theorems 2.5.2 and 2.5.3]{ash}). Therefore we will state our results in terms of the stronger notions of $L^1$ norm and almost uniform convergence. 

\begin{theorem}\label{thm:2}
There is a sequence of feasible solutions to the dual LP (\ref{vlp}) such that its $w$-component converges from above to $I_{X_0}$ in $L^1$ norm and almost uniformly.
\end{theorem}
\begin{proof}
By Theorem~\ref{thm:1}, the optimal solution to the primal is attained by the restriction of the Lebesgue measure to $X_0$. Consequently,
\begin{equation}\label{eq:Thm2_1}
p^* = \int_X I_{X_0}(x) d\lambda(x). 
\end{equation}
By Theorem~\ref{thm:noGap}, there is no duality gap ($p^* = d^*$), and therefore there exists a sequence $(v_k,w_k)\in C^1([0,T]\times X)\times C(X)$ feasible in~(\ref{vlp}) such that
\begin{equation}\label{eq:Thm2_2}
p^* = d^* = \lim_{k\to \infty} \int_{X} w_k(x)\, d\lambda(x).
\end{equation}
From Lemma \ref{lem:v0} we have $w_k\geq 1$ on $X_0$ and
since $w_k \ge 0$ on~$X$ by the fourth constraint of~(\ref{vlp}), we have $w_k \ge I_{X_0}$ on~$X$ for all $k$. Thus, subtracting~(\ref{eq:Thm2_1}) from~(\ref{eq:Thm2_2}) gives
\[
\lim_{k\to\infty}\int_X(w_k(x)-I_{X_0}(x))\,d\lambda(x) = 0,
\]
where the integrand is nonnegative. Hence $w_k$ converges to $I_{X_0}$ in $L^1$ norm. From \cite[Theorems 2.5.2 and 2.5.3]{ash} there exists a subsequence converging almost uniformly.
\end{proof}

\section{LMI relaxations and SOS approximations}\label{sec:LMI}
In this section we show how the infinite-dimensional LP problem~(\ref{rrlp}) can be approximated by a hierarchy of LMI problems with the approximation error vanishing as the relaxation order tends to infinity. The dual LMI problem is a sum-of-squares (SOS) problem and yields a converging sequence of outer approximations to the ROA.

The measures in equation~(\ref{test}) (or (\ref{eq:Liouville})) are fully determined by their values on a family of functions whose span is dense in $C^1([0,T]\times X)$. Hence, since all sets are assumed compact, the family of test functions in~(\ref{test}) can be restricted to any polynomial basis (since polynomials are dense in the space of continuous
functions on compact sets equipped with the supremum norm). The basis of our choice is the set of all monomials. This basis is convenient for subsequent exposition and is employed by existing software (e.g., \cite{glopti, yalmip}). Nevertheless, another polynomial basis may be more appropriate from a numerical point of view (see the Conclusion for a discussion).

Let ${\mathbb R}_k[x]$ denote the vector space of real multivariate polynomials of total degree less than or equal to $k$.
Each polynomial $p(x) \in {\mathbb R}_k[x]$ can be expressed in the monomial basis as
\[
p(x) = \sum_{|\alpha|\leq k} p_{\alpha}x^{\alpha} = \sum_{|\alpha|\leq k} p_{\alpha}(x_1^{\alpha_1}\cdot\ldots\cdot x_n^{\alpha_n}),
\]
where $\alpha$ runs over the multi-indices (vectors of nonnegative integers) such that $|\alpha| = \sum_{i=1}^n \alpha_i \leq k$. A polynomial $p(x)$ is identified with its
vector of coefficients $p:=(p_{\alpha})$ whose entries are indexed by $\alpha$.
Given a vector of real numbers $y:=(y_{\alpha})$ indexed by $\alpha$, we define
the linear functional $L_y : {\mathbb R}_k[x]\to {\mathbb R}$ such that
\[
L_y(p) := p'y := \sum_{\alpha} p_{\alpha} y_{\alpha},
\]
where the prime denotes transposition\footnote{In this paper, the operator prime is also used 
to refer to the adjoint of a linear operator. If the linear operator is real-valued and
finite-dimensional, the adjoint operator coincides with the transposition operator.}.
When entries of $y$ are moments of a measure $\mu$, i.e.,
\[
y_{\alpha} = \int x^{\alpha} d\mu(x),
\]
the linear functional models the integration of a polynomial with respect to $\mu$, i.e.,
\[
L_y(p) = \langle\mu,\:p\rangle = \int p(x)\,d\mu(x) = \sum_{\alpha} p_{\alpha}
\int x^{\alpha} d\mu(x) = p' y.
\]
When this linear functional acts on the square of a polynomial $p$ of degree $k$, it becomes a quadratic form
in the polynomial coefficients space, and we denote by $M_k(y)$ and call the \emph{moment
matrix} of order $k$ the matrix of this quadratic form, which is symmetric and linear in $y$:
\[
L_y(p^2) = p' M_k(y) p.
\]
Finally, given a polynomial $g(x) \in {\mathbb R}[x]$ we define the \emph{localizing matrix} $M_k(g,\,y)$
by the equality
\[
L_y(g p^2) = p' M_k(g,\,y) p.
\]
The matrix $M_k(g,\,y)$ is also symmetric and linear in $y$. For a detailed exposition and examples of moment and localizing matrices see~\cite[Section 3.2.1]{lasserre}.

Let $y$, $y_0$, $y_T$ and $\hat{y}_0$ respectively denote the sequences of moments of measures $\mu$,
$\mu_0$, $\mu_T$ and $\hat{\mu}_0$ such that the constraints in problem (\ref{rrlp}) are satisfied.
Then it follows that these sequences satisfy an infinite-dimensional
linear system of equations corresponding to the equality constraints of problem~(\ref{rrlp}) written explicitly as
\begin{align*}\int_{X_T} v(T,x)\,d\mu_T(x) - \int_X v(0,x)\,d\mu_0(x) - \int_{[0,T]\times X \times U} \hspace{-2em}\mathcal{L}v(t,x,u)\, d\mu(t,x,u)  = 0,\end{align*}
\[\int_X w(x)\,d\mu_0(x) + \int_X w(x)\,d\hat{\mu}_0(x) = \int_X w(x)\,d\lambda(x).\]
For the particular choice of test functions $v(t,x) = t^{\alpha}x^{\beta}$ and $w(x) = x^{\beta}$
for all $\alpha \in {\mathbb N}$ and $\beta \in {\mathbb N}^n$, let us denote by
\[
A_k(y,y_0,y_T,\hat{y}_0) = b_k
\]
the finite-dimensional truncation of this system obtained by considering only the test functions of total degree less than or equal to $2k$. Let further \[\displaystyle {d_X}_i := \Big\lceil \frac{\deg {g^X_i}}{2} \Big\rceil,\:\: \displaystyle {d_U}_i := \Big\lceil \frac{\deg {g^U_i}}{2} \Big\rceil,\:\: \displaystyle {d_T}_i := \Big\lceil \frac{\deg {g^{X_T}_i}}{2} \Big\rceil,\]
where $\mathrm{deg}$ denotes the degree of a polynomial. The primal LMI relaxation of order $k$ then reads
\begin{equation}\label{plmi}
\begin{array}{rcllll}
p^*_k & =  &\max & (y_0)_0 \\
&& \mathrm{s.t.} & A_k(y,y_0,y_T,\hat{y}_0) = b_k \\
&&& \hspace{0em} M_k(y) \succeq 0, & \hspace{0em}M_{k-{d_X}_i}({g^X_i},y) \succeq 0, &i\in\mathbb{Z}_{[1,n_X]}\\
&&& \hspace{0em} M_{k-1}(t(T\!-\!t),y) \succeq 0, &  \hspace{0em}M_{k-{d_U}_i}(g^U_i,y) \succeq 0, &i\in\mathbb{Z}_{[1,n_U]} \\
&&& \hspace{0em} M_k(y_0) \succeq 0, &\hspace{0em} M_{k-{d_X}_i}({g^X_i},y_0) \succeq 0, &i\in\mathbb{Z}_{[1,n_X]} \\
&&& \hspace{0em} M_k(y_T) \succeq 0, &\hspace{0em} M_{k-{d_T}_i}({g^{X_T}_i},y_T) \succeq 0, &i\in\mathbb{Z}_{[1,n_T]}\\
&&& \hspace{0em} M_k(\hat{y}_0) \succeq 0, & \hspace{0em} M_{k-{d_X}_i}({g^X_i},\hat{y}_0) \succeq 0, &i\in\mathbb{Z}_{[1,n_X]}, 
\end{array}
\end{equation}
where the notation $\succeq 0$ stands for positive semidefinite
and the minimum is over sequences $(y, y_0, y_T,\hat{y}_0)$ truncated to degree $2k$. The objective function is the first element (i.e., the mass) of the truncated moment sequence $y_0$ corresponding to the initial measure; the equality constraint captures the two equality constraints of problem~(\ref{rrlp}) evaluated on monomials of degree up to $2k$; and the LMI constraints involving the moment and localizing matrices capture the nonnegativity and support constraints on the measures, respectively. Note that both the equality constraint and the LMI constraints are necessarily satisfied by the moment sequence of any vector of measures feasible in (\ref{rrlp}). The constraint set of~(\ref{plmi}) is therefore looser than that of (\ref{rrlp}); however, the discrepancy between the two constraint sets monotonically vanishes as the relaxation order $k$ tends to infinity (see Corollary~\ref{cor:pdconv} below).

Problem (\ref{plmi}) is a semidefinite program (SDP),
where a linear function is minimized subject to convex LMI
constraints, or equivalently a finite-dimensional LP in the cone of positive semidefinite matrices.

Without loss of generality we make the following standard assumption for the reminder of this section.

\begin{assumption}\label{compact}
One of the polynomials defining the sets $X$, $U$ respectively $X_T$,
is equal to ${g^X_i}(x) = R_X-\|x\|^2_2$, ${g^U_i}(u) = R_U-\|u\|^2_2$ respectively ${g^{X_T}_i}(x)=R_T-\|x\|^2_2$ for some constants $R_X\ge0$, $R_U\ge0$ respectively $R_T\ge0$.
\end{assumption}

Assumption \ref{compact} is without loss of generality since the sets $X$, $U$ and $X_T$ are bounded, and therefore redundant ball constraints of the form $R_X-\|x\|^2_2 \ge0$, $R_U-\|u\|^2_2\ge0$ and $R_T-\|x\|^2_2\ge0$ can always be added to the description of the sets $X$, $U$ and $X_T$ for sufficiently large $R_X$, $R_U$ and~$R_T$.

\looseness-1
The dual to the SDP problem (\ref{plmi}) is given by
\begin{equation}\label{dlmi}
\begin{array}{rcll}
d^*_k & = & \inf & {w}' l \\\vspace{0.5mm}
&& \mathrm{s.t.} &  -\mathcal{L}v(t,x,u)
= p(t,x,u) + q_0(t,x,u) t(T-t) \\\vspace{1mm}
&&&  + \sum_{i=1}^{n_X} q_i(t,x,u) {g^X_i}(x) + \sum_{i=1}^{n_U} r_i(t,x,u) {g^U_i}(x) \\ \vspace{1mm}
&&& w(x)-v(0,x)-1 = p_0(x) + \sum_{i=1}^{n_X} {q_0}_i(x) {g^X_i}(x) \\\vspace{1mm}
&&& v(T,x) = p_T(x) + \sum_{i=1}^{n_T} {q_T}_i(x) {g^{X_T}_i}(x) \\
&&& w(x) = s_0(x) + \sum_{i=1}^{n_X} {s_0}_i(x) {g^X_i}(x),
\end{array}
\end{equation}
where $l$ is the vector of the moments of the Lebesgue measure over $X$ indexed in the same basis in which the polynomial $w(x)$ with coefficients $w$ is expressed. The minimum is over polynomials $v(t,x) \in {\mathbb R}_{2k}[t,x]$ and $w\in \mathbb{R}_{2k}[x]$, and polynomial sum-of-squares $p(t,x,u)$, $q_i(t,x,u)$, $i\in\mathbb{Z}_{[0,n_X]}$, $r_i(t,x,u)$, $i\in\mathbb{Z}_{[1,n_U]}$, $p_0(x)$, $p_T(x)$, ${q_0}_i(x)$, ${q_T}_i(x)$, $s_0(x)$, ${s_0}_i(x)$, $i\in\mathbb{Z}_{[1,n_X]}$ of appropriate degrees. The constraints that polynomials are sum-of-squares can be written explicitly as LMI constraints~(see, e.g., \cite{lasserre}), and the objective is linear in the coefficients of the polynomial $w(x)$; therefore problem~(\ref{dlmi}) can be formulated as an SDP.

\begin{theorem}\label{lem:noGapRelax}
There is no duality gap between primal LMI problem (\ref{plmi}) and dual LMI problem (\ref{dlmi}),
i.e., $p^*_k = d^*_k$.
\end{theorem}
\begin{proof}
See Appendix~D.
\end{proof}

\section{Outer approximations and convergence results}\label{sec:convResults}
In this section we show how the dual LMI problem~(\ref{dlmi}) gives rise to a sequence of outer approximations to the ROA $X_0$ with a guaranteed convergence. In addition, we prove the convergence of the primal and dual optimal values  $p_k^*$ and $d_k^*$ to the volume of the ROA, and the convergence of the $w$-component of an optimal solution to the dual LMI problem~(\ref{dlmi}) to the indicator function of the ROA $I_{X_0}$.

Let the polynomials $(w_k, v_k)$, each of total degree at most $2k$, denote an optimal solution to the problem $k^{\mathrm{th}}$ order dual SDP approximation~(\ref{dlmi}) and let $\bar{w}_k := \min_{i\le k}w_i$ and $\bar{v}_k := \min_{i\le k}v_i$ denote their running minima. Then, in view of Lemma~\ref{lem:v0} and the fact that any feasible solution to (\ref{dlmi}) is feasible in (\ref{vlp}), the sets
\begin{equation}\label{eq:X0k}
{X_0}_k := \{x\in X : v_k(0,x)\ge 0 \}
\end{equation}
and
\begin{equation}\label{eq:barX0k}
  \bar{X}_{0k} := \{x\in X : \bar{v}_k(0,x)\ge 0 \}
\end{equation}
provide outer approximations to the ROA; in fact, the inclusions ${X_0}_k\supset \bar{X}_{0k} \supset X_0$ hold for all $k\in\{1,2,\ldots\}$.

Our first convergence result proves the convergence of $w_k$ and $\bar{w}_k$ to the indicator function of the ROA $I_{X_0}$.

\begin{theorem}\label{thm:dualConvFun}
Let $w_k \in {\mathbb R}_{2k}[x]$ denote the $w$-component of an optimal solution to the dual LMI problem (\ref{dlmi}) and let $\bar{w}_k(x) =\min_{i\le k} w_i(x)$. Then $w_k$ converges from above to $I_{X_0}$ in $L^1$ norm and  $\bar{w}_k$ converges from above to $I_{X_0}$ in $L^1$ norm and almost uniformly.
\end{theorem}
\begin{proof} From Lemma~\ref{lem:v0} and Theorem~\ref{thm:2}, for every $\epsilon > 0$ there exists a $(v,w)\in C^1([0,T]\times X)\times C(X)$ feasible in~(\ref{vlp}) such that $w \ge I_{X_0}$ and $\int_X (w - I_{X_0})\,d\lambda < \epsilon$. Set 
\begin{align*}
	\tilde{v}(t,x) &:= v(t,x) - \epsilon t + (T+1)\epsilon,\\
	\tilde{w}(x) &:= w(x) + (T+3)\epsilon.
\end{align*}
Since $v$ is feasible in (\ref{vlp}), we have $\mathcal{L}\tilde{v} = \mathcal{L}v - \epsilon$, and $\tilde{v}(T,x) = v(T,x) + \epsilon$. Since also $\tilde{w}(x) -\tilde{v}(0,x)\ge 1+2\epsilon$, it follows that $(\tilde{v},\tilde{w})$ is \emph{strictly} feasible in~(\ref{vlp}) with a margin at least $\epsilon$. Since $[0,T]\times X$ and $X$ are compact, there exist\footnote{This follows from an extension of the Stone-Weierstrass theorem that allows for a simultaneous uniform approximation of a function and its derivatives by a polynomial on a compact set; see, e.g.,~\cite{hirsch}. } polynomials $\hat{v}$ and $\hat{w}$ of a sufficiently high degree such that $\sup_{[0,T]\times X}\limits|\tilde{v} - \hat{v}|< \epsilon$, $\sup_{[0,T]\times X\times U}\limits |\mathcal{L}\tilde{v} - \mathcal{L}\hat{v}| < \epsilon$ and $\sup_X |\tilde{w}-\hat{w}| < \epsilon$. The pair of polynomials $(\hat{v},\hat{w})$ is therefore \emph{strictly} feasible in~(\ref{vlp}) and as a result, under Assumption~\ref{compact}, feasible in~(\ref{dlmi}) for a sufficiently large relaxation order $k$ (this follows from the classical Positivstellensatz by Putinar; see, e.g., \cite{lasserre} or \cite{putinar}), and moreover $\hat{w}\ge w$. Consequently, $\int_X| \tilde{w} - \hat{w}|\,d\lambda \le \epsilon\lambda(X)$, and so $\int_X( \hat{w} - w)\,d\lambda \le \epsilon\lambda(X)(T+4)$. Therefore
\[\int_X  ( \hat{w} - I_{X_0})\,d\lambda < \epsilon K,\quad \hat{w}\ge I_{X_0},\]
where $K:=[1 + (T+4)\lambda(X)] < \infty$ is a constant. This proves the first statement since $\epsilon$ was arbitrary.

The second statement immediately follows since given a sequence $w_k \to I_{X_0}$ in $L_1$ norm, there exists a subsequence $w_{k_i}$ that converges almost uniformly to $I_{X_0}$ by \cite[Theorems 2.5.2 and 2.5.3]{ash} and clearly $\bar{w}_k(x) \le \min\{w_{k_i}(x)\: : \: k_i \le k\}$.
\end{proof}

The following Corollary follows immediately from Theorem~\ref{thm:dualConvFun}. 
\begin{corollary}\label{cor:pdconv}
The sequence of infima of LMI problems~(\ref{dlmi}) converges monotonically from above to the supremum of the infinite-dimensional LP problem~(\ref{vlp}), i.e., $d^*\le d_{k+1}^* \le d_k^*$ and $\lim_{k\to\infty} d_k^* = d^*$. Similarly, the sequence of maxima of LMI problems (\ref{plmi}) converges monotonically from above to the maximum of the infinite-dimensional LP problem (\ref{rlp}), i.e., $p^* \le p^*_{k+1}\le p_k^*$ and $\lim_{k\to \infty} p^*_k = p^*$.
\end{corollary}
\begin{proof} Monotone convergence of the dual optima $d_k^*$ follows immediately from Theorem~\ref{thm:dualConvFun} and from the fact that the higher the relaxation order $k$, the looser the constraint set of the minimization problem~(\ref{dlmi}). To prove convergence of the primal maxima observe that from weak SDP duality we have $d_k^* \ge p_k^*$ and from Theorems~\ref{thm:dualConvFun} and~\ref{thm:noGap} it follows that $d_k^* \to d^* = p^*$. In addition, clearly $p_k^* \ge p^*$ and $p_{k+1}^* \le p_k^*$ since the higher the relaxation order $k$, the tighter the constraint set of the maximization problem~(\ref{plmi}). Therefore $p_k^*\to p^*$ monotonically from above. \end{proof}

Theorem~\ref{thm:dualConvFun} establishes a functional convergence of $w_k$ to $I_{X_0}$ and Corollary~\ref{cor:pdconv} a convergence of the primal and dual optima $p_k^*$ and $d_k^*$ to the volume of the ROA $\lambda(X_0)=p^*=d^*$. Finally, the following theorem establishes a set-wise convergence of the sets~(\ref{eq:X0k}) and~(\ref{eq:barX0k}) to the ROA $X_0$.


\begin{theorem}\label{thm:convSet}
Let $(v_k,w_k) \in {\mathbb R}_{2k}[t,x]\times {\mathbb R}_{2k}[x]$ denote a solution to the dual LMI problem (\ref{dlmi}). Then the sets $X_{0k}$ and $\bar{X}_{0k}$ defined in~(\ref{eq:X0k}) and (\ref{eq:barX0k}) converge to the ROA $X_0$ from the outside such that
$X_{0k}\supset \bar{X}_{0k}\supset X_0$ and
\[
 \lim_{k\to\infty}\lambda(X_{0k}\setminus X_0) = 0 \:\:\:\:\text{and}\:\:\:\:
 \lim_{k\to\infty}\lambda(\bar{X}_{0k}\setminus X_0) = 0.
\]
Moreover the convergence of $\bar{X}_{0k}$ is monotonous, i.e., $\bar{X}_{0i} \subset \bar{X}_{0j}$ whenever $i \ge j$.
\end{theorem}
\begin{proof} The inclusion $X_{0k}\supset \bar{X}_{0k} \supset X_0$ follows from Lemma~\ref{lem:v0} since any solution to (\ref{dlmi}) is feasible in (\ref{vlp}) and since $\bar{X}_{0k}=\cap_{i=1}^k X_{0i}$. The latter fact also proves the monotonicity of the sequence $\bar{X}_{0k}$. Next, from Lemma~\ref{lem:v0} we have $w_k \ge I_{X_0}$ and therefore, since $w_k\ge v_k(0,\cdot)+1$ on $X$, we have $w_k\ge I_{X_{0k}} \ge I_{\bar{X}_{0k}} \ge I_{X_{0}}$ on $X$. In addition, from Theorem~\ref{thm:dualConvFun} we have $w_k \to I_{X_0}$ in $L^1$ norm on $X$; therefore 
\begin{align*}
\lambda(X_0) = \int_X I_{X_0}\,d\lambda &=  \lim_{k\to\infty} \int_X w_k\,d\lambda  \ge  \lim_{k\to\infty} \int_X I_{X_{0k}}\,d\lambda\\  & =  \lim_{k\to\infty}\lambda(X_{0k}) \ge \lim_{k\to\infty}\lambda(\bar{X}_{0k}).
\end{align*}
But since $X_0 \subset \bar{X}_{0k} \subset X_{0k} $ we must have $\lambda(X_0) \le \lambda(\bar{X}_{0k}) \le \lambda(X_{0k}) $ and the theorem follows.
\end{proof}

\section{Free final time}\label{sec:infiniteTime}
In this section we outline a  straightforward extension of our approach to the problem of reaching the target set $X_T$ at \emph{any} time before $T < \infty$ (and not necessarily staying in $X_T$ afterwards).

It turns out that the set of all initial states $x_0$  from which it is possible to reach $X_T$ at a time $t\leq T$ can be obtained as the support of an optimal solution $\mu_0^*$ to the problem
\begin{equation}\label{rlp_free}
\begin{array}{rcll}
 & \sup & \mu_0(X) \\
& \mathrm{s.t.} & \mu_T = \mu_0\otimes \delta_0 + {\mathcal L}'\mu \\
&& \mu\geq 0,\: \lambda\geq\mu_0\geq 0, \: \mu_T\geq 0\\
&& \mathrm{spt}\:\mu \subset [0,T]\times X\times U\\
&& \mathrm{spt}\:\mu_0 \subset X, \:\:\mathrm{spt}\:\mu_T \subset [0,T]\times X_T,
\end{array}
\end{equation}
where the supremum is over a vector of nonnegative
measures $(\mu,\mu_0,\mu_T) \in M([0,T]\times X\times U)\times M(X)\times M([0,T]\times X_T)$. Note that the only difference to problem~(\ref{rlp}) is in the support constraints of the final measure $\mu_T$.

The dual to this problem reads as
\begin{equation}\label{vvlp}
\begin{array}{rclll}
 & \inf & \displaystyle\int_{X} w(x)\, d\lambda(x) \\
& \mathrm{s.t.} & \mathcal{L}v(t,x,u) \leq 0, \:\: &\!\!\forall\, (t,x,u) \in [0,T]\times X\times U \\
&& w(x) \ge v(0,x) + 1, \:\: &\!\!\forall\, x \in X \\
&& v(t,x) \geq 0, \:\: &\!\!\forall\, (t,x) \in [0,T] \times X_T\\
&& w(x) \geq 0, \:\: &\!\!\forall\, x \in X.
\end{array}
\end{equation}
The only difference to problem~(\ref{vlp}) is in the third constraint which now requires that $v(t,x)$ is nonnegative on $X_T$ for \emph{all} $t\in[0,T]$.

All results from the previous sections hold with proofs being almost verbatim copies.

\section{Numerical examples}\label{sec:NumeEx}
In this section we present five examples of increasing complexity to illustrate our approach: a univariate uncontrolled cubic system, the Van der Pol oscillator, a double integrator, the Brockett integrator and an acrobot. For numerical implementation,  one can either use Gloptipoly~3~\cite{glopti} to formulate the primal problem on measures and then extract the dual solution provided by a primal-dual SDP solver or formulate directly the dual SOS problem using, e.g., YALMIP~\cite{yalmip} or SOSTOOLS~\cite{sostools}. As an SDP solver we used SeDuMi~\cite{sedumi} for the first three examples and MOSEK for the last two examples. For computational purposes the problem data should be scaled such that the constraint sets are contained in, e.g., unit boxes or unit balls; in particular the time interval $[0,T]$ should be scaled to $[0,1]$ by multiplying the vector field $f$ by $T$. Computational aspects are further discussed in the Conclusion.

Whenever the approximations $X_{0k}$ defined in~(\ref{eq:X0k}) are monotonous (which is not guaranteed) we report these approximations (since then they are equal to the monotonous version $\bar{X}_{0k}$ defined in~(\ref{eq:barX0k})); otherwise we report $\bar{X}_{0k}$.

\subsection{Univariate cubic dynamics}

Consider the system given by
\[\dot{x} = x(x-0.5)(x+0.5), \]
the constraint set $X = [-1,1]$, the final time $T = 100$ and the target set $X_T = [-0.01,0.01]$. The ROA can in this case be determined analytically as $X_0 = [-0.5,0.5]$. Polynomial approximations to the ROA for degrees $d\in\{4,8,16,32\}$ are shown in Figure~\ref{fig:1}. As expected the functional convergence of the polynomials to the discontinuous indicator function is rather slow; however, the set-wise convergence of the approximations $X_{0k}$ (where $k = d/2$) is very fast as shown in Table~\ref{tab:1}. Note that the volume error is not monotonically decreasing -- indeed what is guaranteed to decrease is the integral of the approximating polynomial $w(x)$, not the volume of $X_{0k}$. Taking the monotonically decreasing approximations $\bar{X}_{0k}$ defined in~(\ref{eq:barX0k}) would prevent the volume increase. Numerically, a better behavior is expected when using alternative polynomial bases (e.g., Chebyshev polynomials)
instead of the monomials; see the conclusion for a discussion.

\begin{table}[ht]
\centering
\caption{\rm Univariate cubic dynamics -- relative volume error of the outer approximation to the ROA $X_0 = [-0.5,0.5]$ as a function of the approximating polynomial degree.}\label{tab:1}\vspace{2mm}
\begin{tabular}{ccccc}
\toprule
degree & 4 & 8 & 16 & 32 \\\midrule
error & 31.60\,\% & 3.31\,\% & 0.92\,\% & 1.49\,\% \\
\bottomrule
\end{tabular}
\end{table}

\begin{figure*}[th]
\begin{picture}(140,135)

	\put(-20,10){\includegraphics[width=45mm]{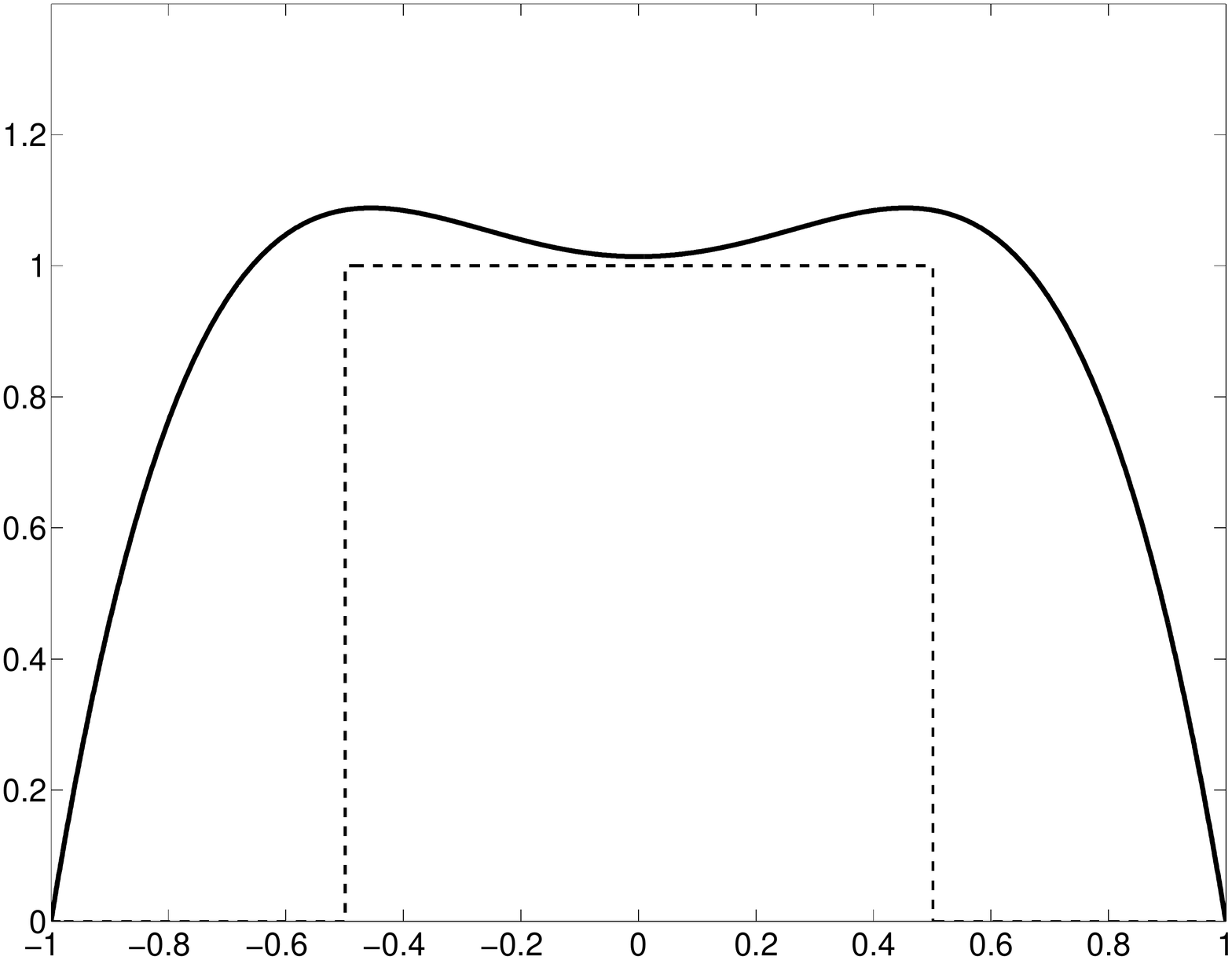}} 
	\put(110,10){\includegraphics[width=45mm]{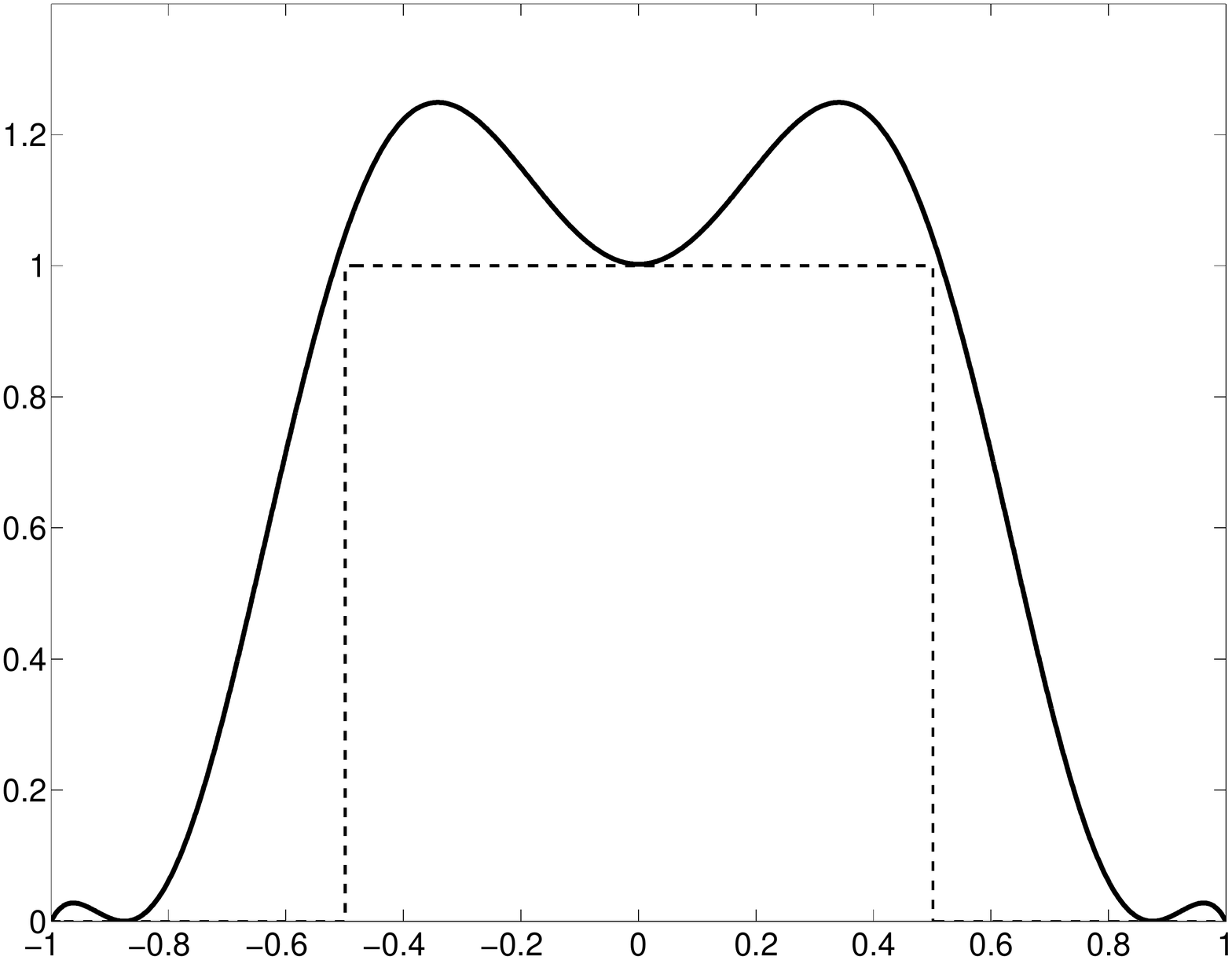}}
	\put(238,10){\includegraphics[width=45mm]{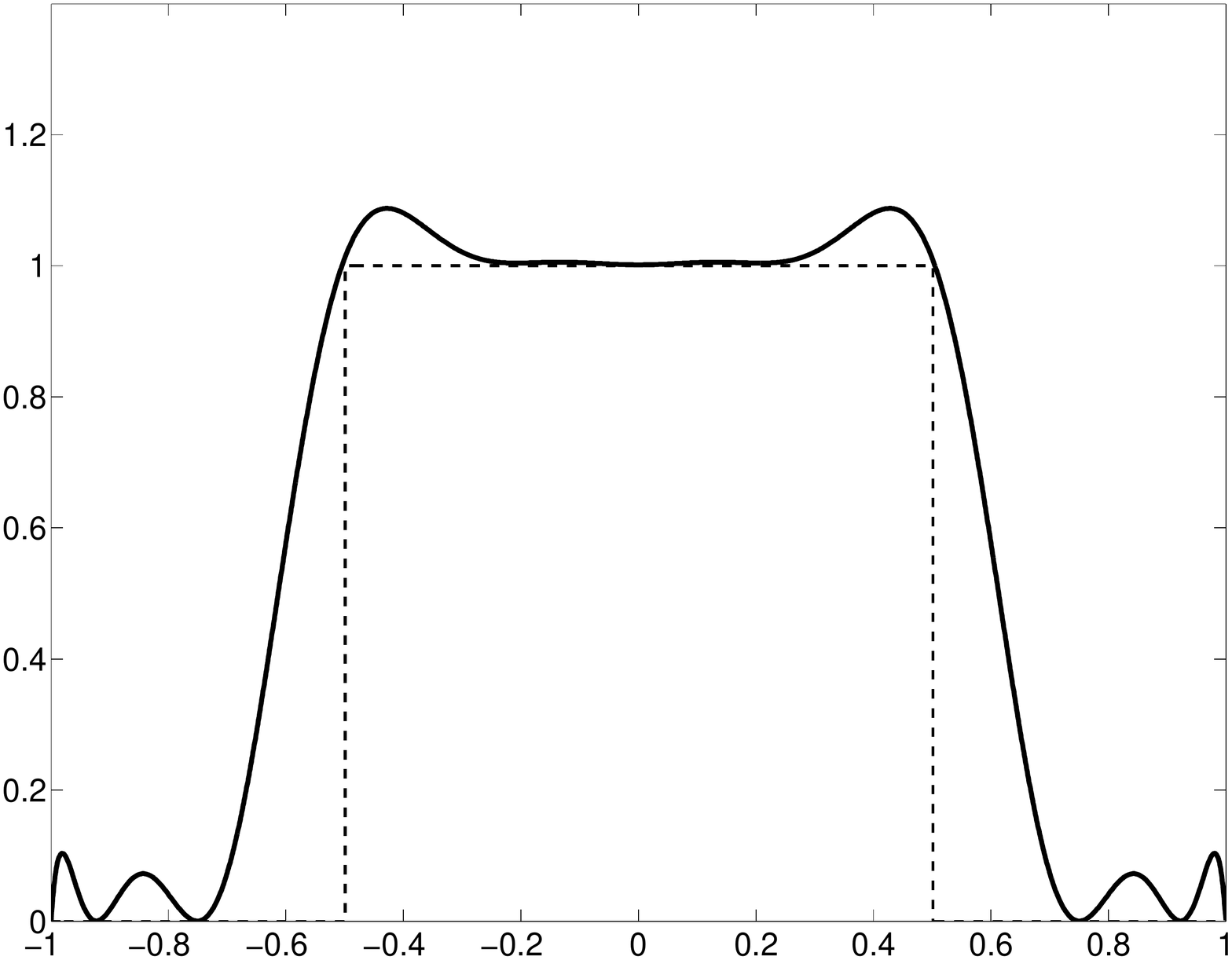}}
	\put(365,10){\includegraphics[width=45mm]{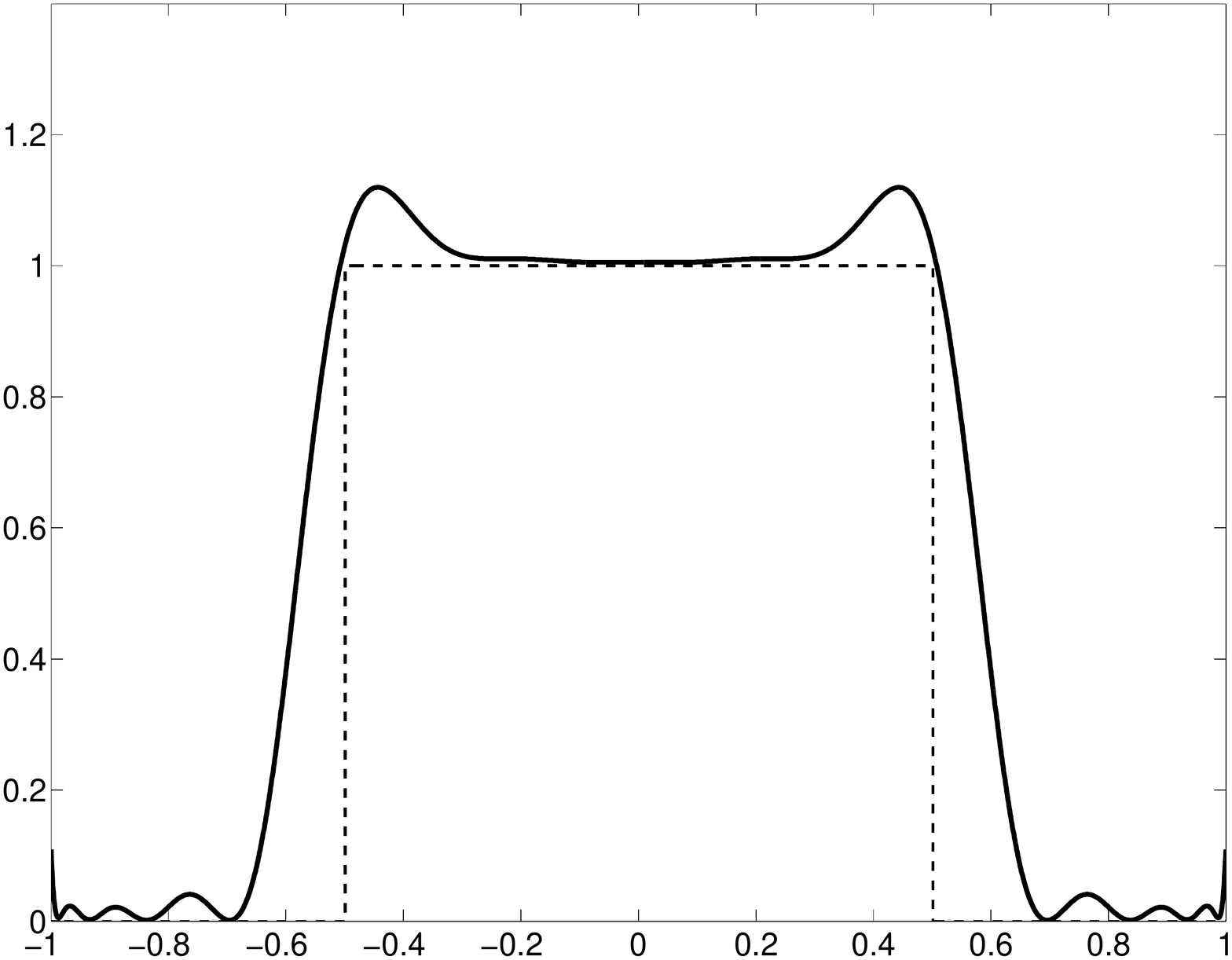}}

\put(43,0){\footnotesize $x$}
\put(173,0){\footnotesize $x$}
\put(301,0){\footnotesize $x$}
\put(428,0){\footnotesize $x$}

	\put(-5,96){\footnotesize $d = 4$}
	\put(125,96){\footnotesize $d = 8$}
	\put(253,96){\footnotesize $d = 16$}
	\put(380,96){\footnotesize $d = 32$}

\end{picture}
\caption{Univariate cubic dynamics -- polynomial approximations (solid line) to the ROA indicator function
$I_{X_0} =I_ {[-0.5,0.5]}$ (dashed line) for degrees $d\in\{4,8,16,32\}$.}
\label{fig:1}
\end{figure*}

\subsection{Van der Pol oscillator}

%
%

%
%

As a second example consider a scaled version of the uncontrolled reversed-time Van~der~Pol oscillator given by
\begin{align*}
	\dot{x}_1 &= -2x_2,\\
	\dot{x}_2 &= 0.8x_1 + 10(x_1^2-0.21)x_2.
\end{align*}
The system has one stable equilibrium at the origin with a bounded region of attraction \[X_0\subset X:= [-1.2,\,1.2]^2.\] In order to compute an outer approximation to this region we take $T=100$ and $X_T = \{x\: : \: \|x\|_2\le 0.01\}$. Plots of the ROA estimates $X_{0k}$ for $d = 2k \in \{10,12,14,16\}$ are shown in Figure~\ref{fig:2}. We observe a relatively-fast convergence of the super-level sets to the ROA -- this is confirmed by the relative volume error\footnote{The relative volume error was computed approximately by Monte Carlo integration.} summarized in Table~\ref{tab:2}. Figure~\ref{fig:vp3D} then shows the approximating polynomial itself for degree $d = 18$.
Here too, a better convergence is expected if instead of monomials, a more appropriate polynomial basis is used.

\begin{table}[!ht]
\centering
\caption{\rm Van der Pol oscillator -- relative error of the outer approximation to the ROA $X_0$ as a function of the approximating polynomial degree.} \label{tab:2} \vspace{2mm}
\begin{tabular}{ccccc}
\toprule
degree & 10 & 12 & 14 & 16 \\\midrule
error & 49.3\,\% & 19.7\,\% & 11.1\,\% & 5.7\,\% \\
\bottomrule
\end{tabular}
\end{table}

\begin{figure*}[ht]
	\begin{picture}(140,105)
	\put(-20,10){\includegraphics[width=45mm]{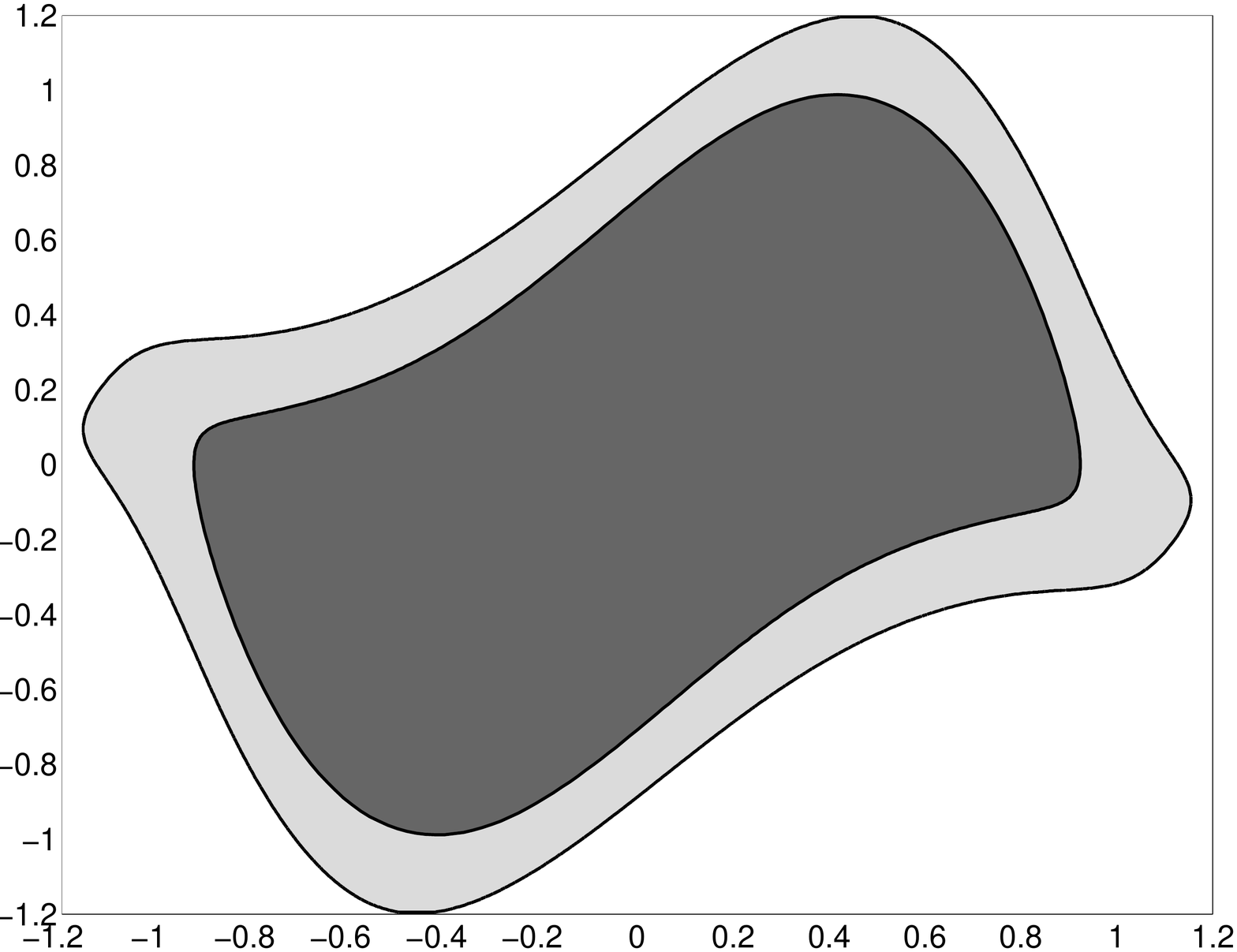}} 
	\put(110,10){\includegraphics[width=45mm]{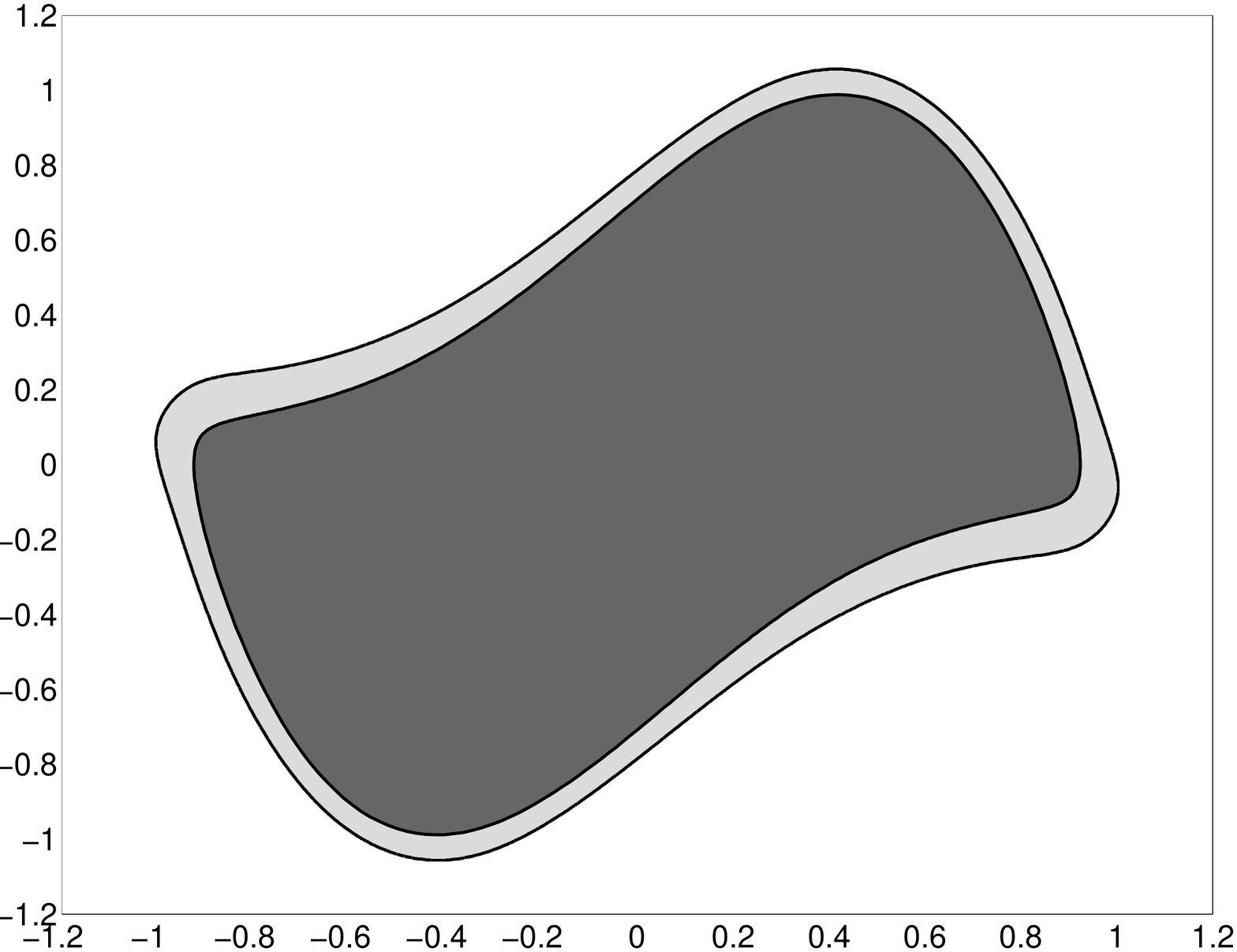}}
	\put(238,10){\includegraphics[width=45mm]{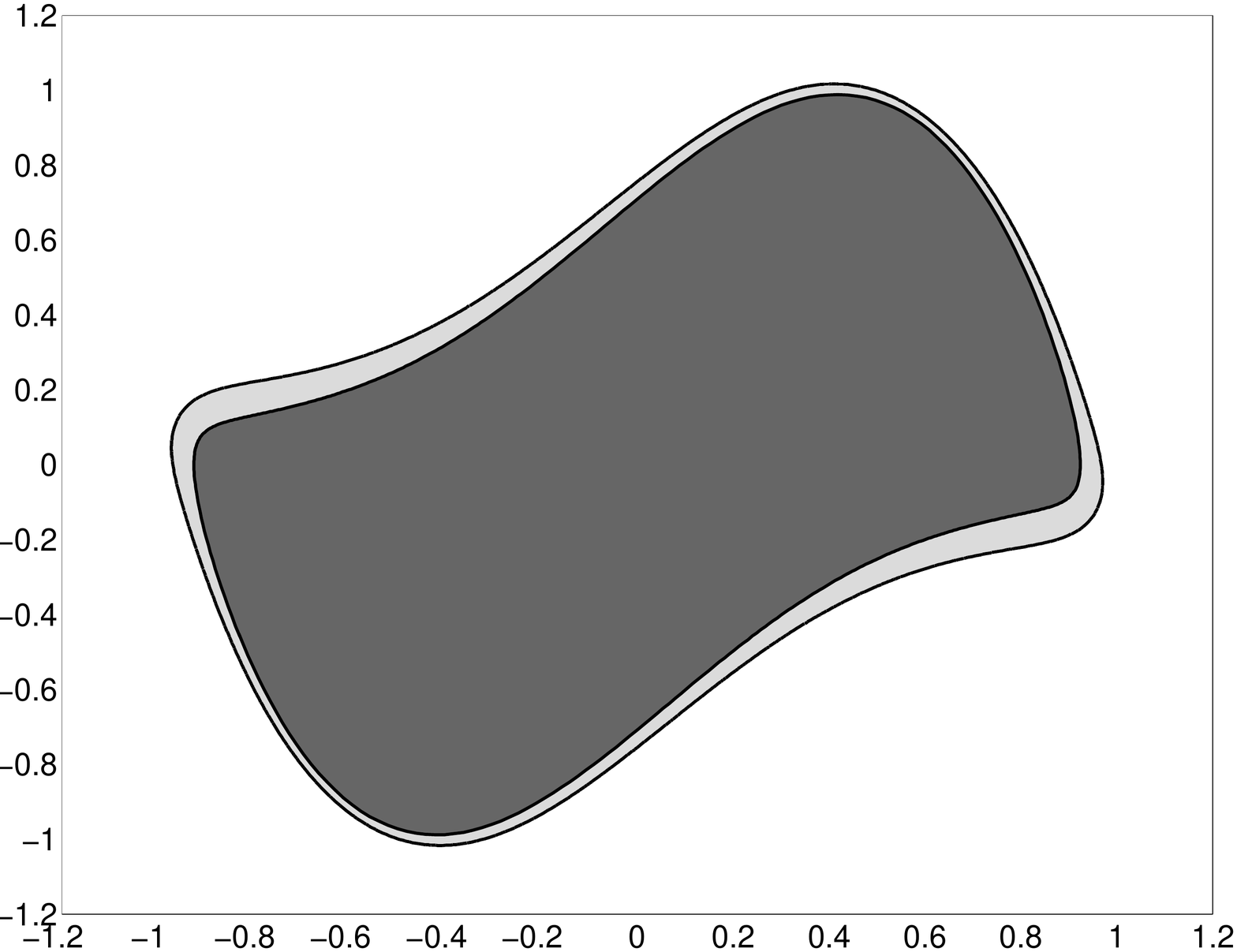}}
	\put(365,10){\includegraphics[width=45mm]{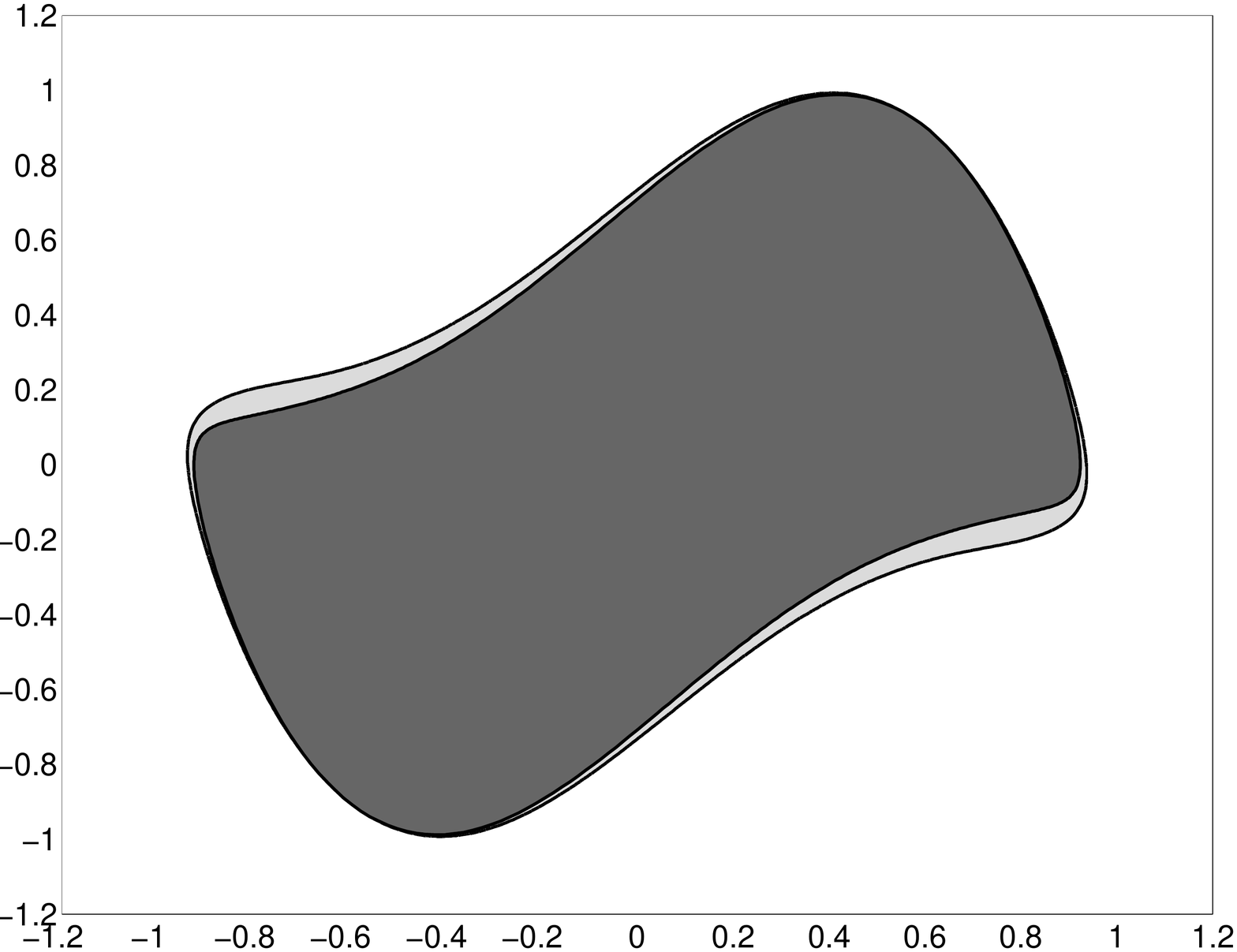}}

	\put(0,94){\footnotesize $d = 10$}
	\put(130,94){\footnotesize $d = 12$}
	\put(258,94){\footnotesize $d = 14$}
	\put(385,94){\footnotesize $d = 16$}
	\end{picture}
	\caption{Van der Pol oscillator -- semialgebraic outer approximations (light gray) to the ROA (dark gray) for degrees $d\in\{10,12,14,16\}$.}
	\label{fig:2}
\end{figure*}

\begin{figure}[h]
  \centering
\includegraphics[width=70mm]{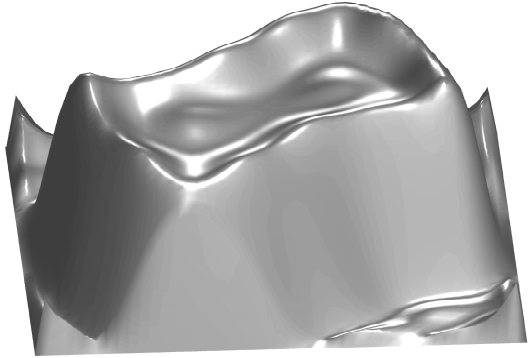}
    \caption{Van der Pol oscillator  -- a polynomial approximation of degree 18 of the ROA indicator function $I_{X_0}$.}
    \label{fig:vp3D}
\end{figure}

\subsection{Double integrator}

To demonstrate our approach in a controlled setting we first consider a double integrator
\begin{align*}
	\dot{x}_1 &= x_2 \\
	\dot{x}_2 & = u.	
\end{align*}
The goal is to find an approximation to the set of all initial states $X_0$ that can be steered to the origin at\footnote{In this case, the sets of all initial states that can be steered to the origin \emph{at} time $T$ and at any time \emph{before} $T$ are the same. Therefore we could also use the free-final-time approach of Section~\ref{sec:infiniteTime}.} the final time $T = 1$. Therefore we set $X_T = \{0\}$ and the constraint set such that $X_0 \subset X$, e.g., $X = [-0.7,0.7]\times [-1.2,1.2]$. The solution to this problem can be computed analytically as
\[X_0 = \{x \: : \: V(x) \le 1 \},\]
where
\[
 V(x) = \begin{cases}
 	x_2 + 2\sqrt{x_1+\frac{1}{2}x_2^2} & \mbox{if} \quad x_1 + \frac{1}{2}x_2|x_2| > 0, \\
	-x_2 + 2\sqrt{-x_1+\frac{1}{2}x_2^2} & \mbox{otherwise}.
 \end{cases}
\]
The ROA estimates $X_{0k}$ for $d=2k\in\{6,8,10,12\}$ are shown in  Figure~\ref{fig:3}; again we observe a relatively fast convergence of the super-level set approximations, which is confirmed by the relative volume errors in Table~\ref{tab:3}.

\begin{table}[ht]
\centering
\caption{\rm Double integrator -- relative error of the outer approximation to the ROA $X_0$ as a function of the approximating polynomial degree.}\label{tab:3}\vspace{2mm}
\begin{tabular}{ccccc}
\toprule
degree & 6 & 8 & 10 & 12 \\\midrule
error & 75.7\,\% & 32.6\,\% & 21.2\,\% & 16.0\,\% \\
\bottomrule
\end{tabular}
\end{table}

\begin{figure*}[!h]
\begin{picture}(140,135)
	\put(-20,10){\includegraphics[width=45mm]{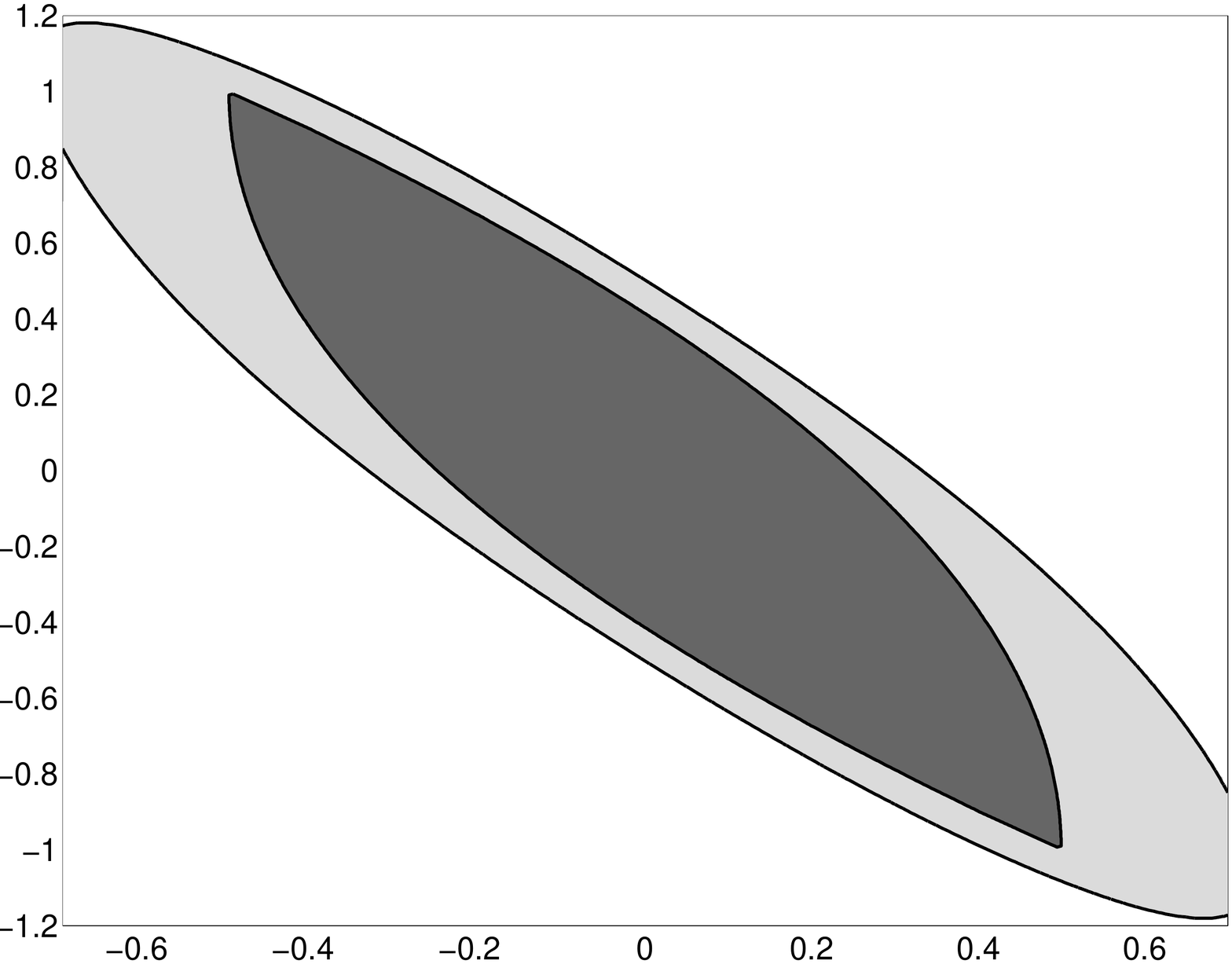}} 
	\put(110,10){\includegraphics[width=45mm]{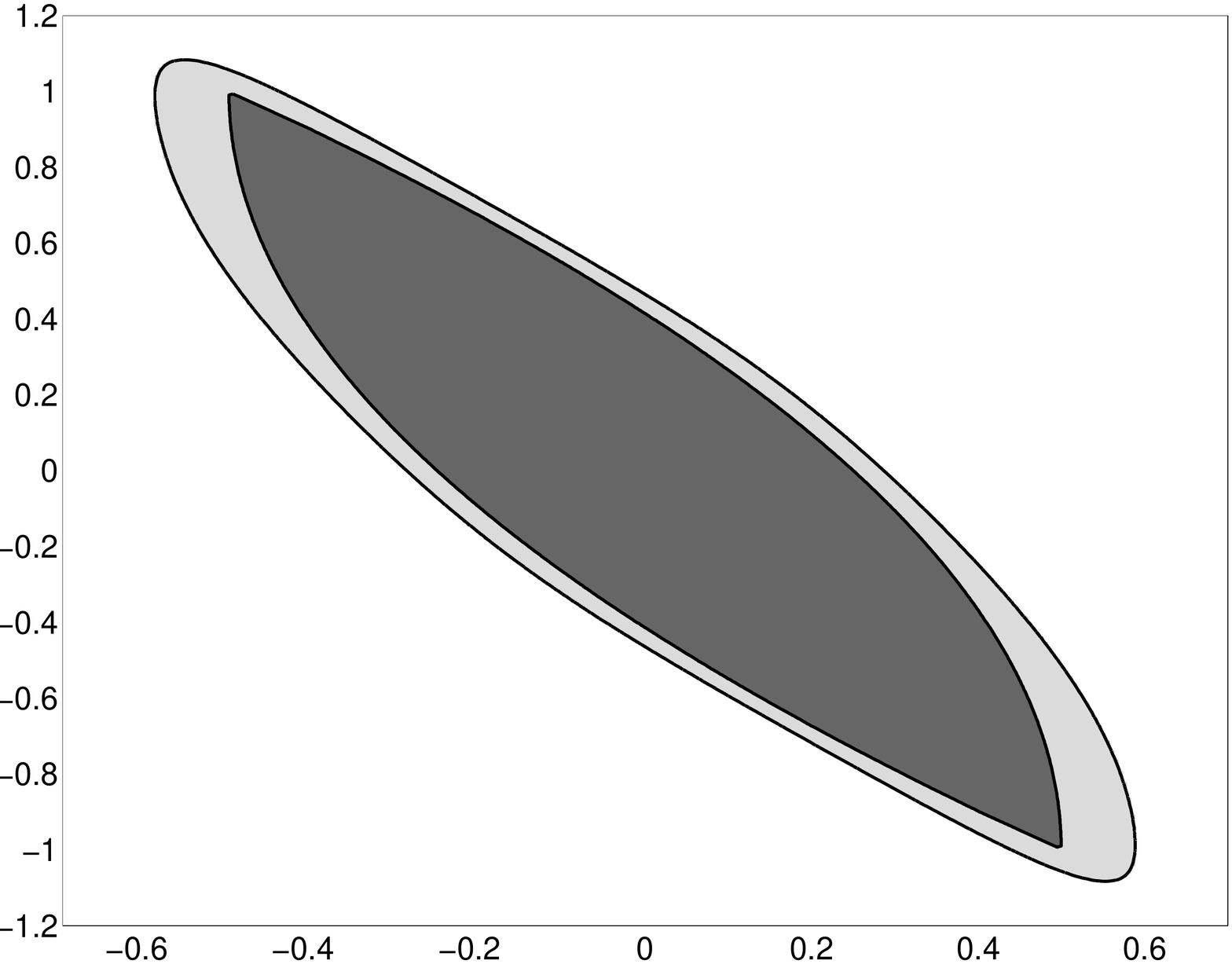}}
	\put(238,10){\includegraphics[width=45mm]{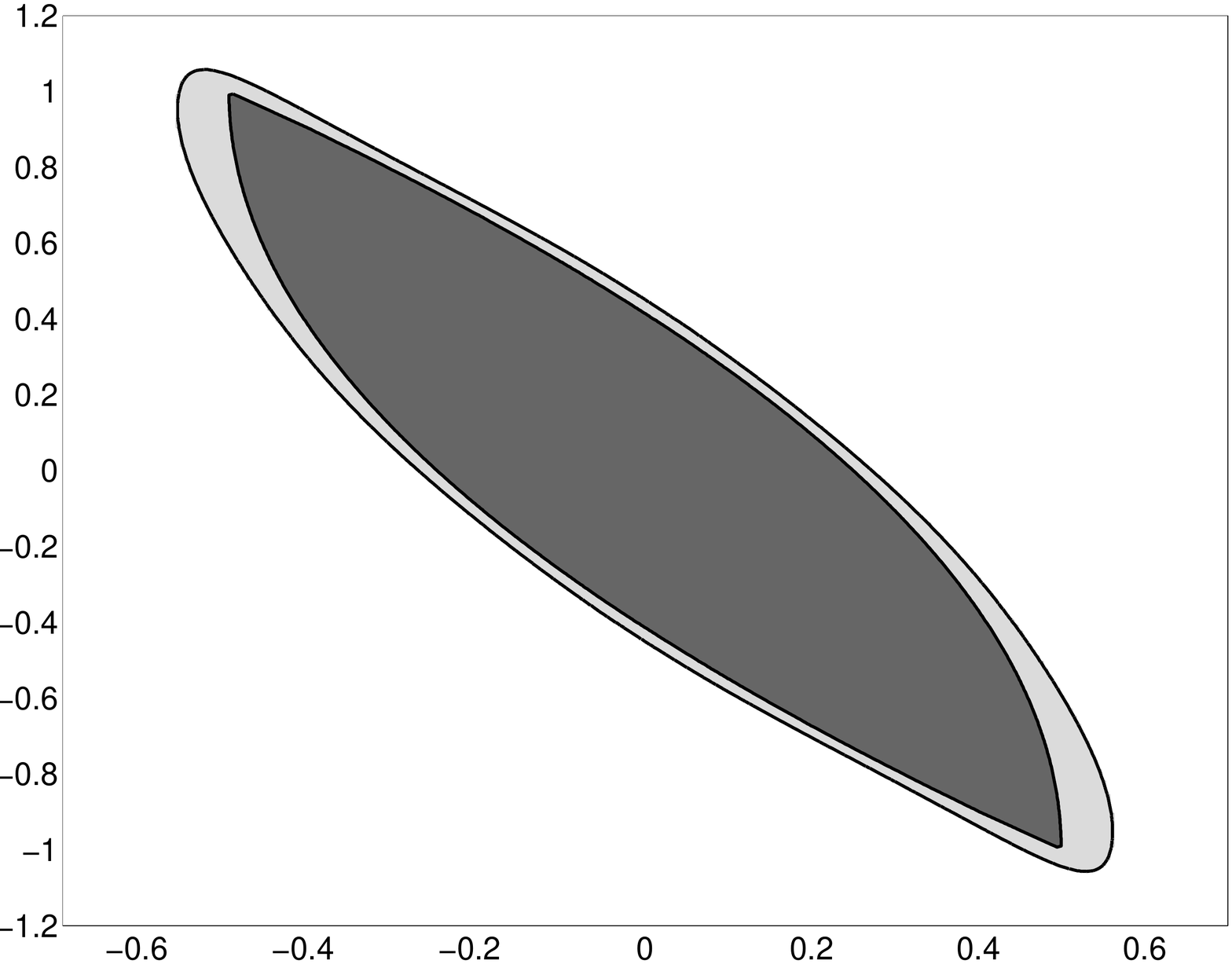}}
	\put(365,10){\includegraphics[width=45mm]{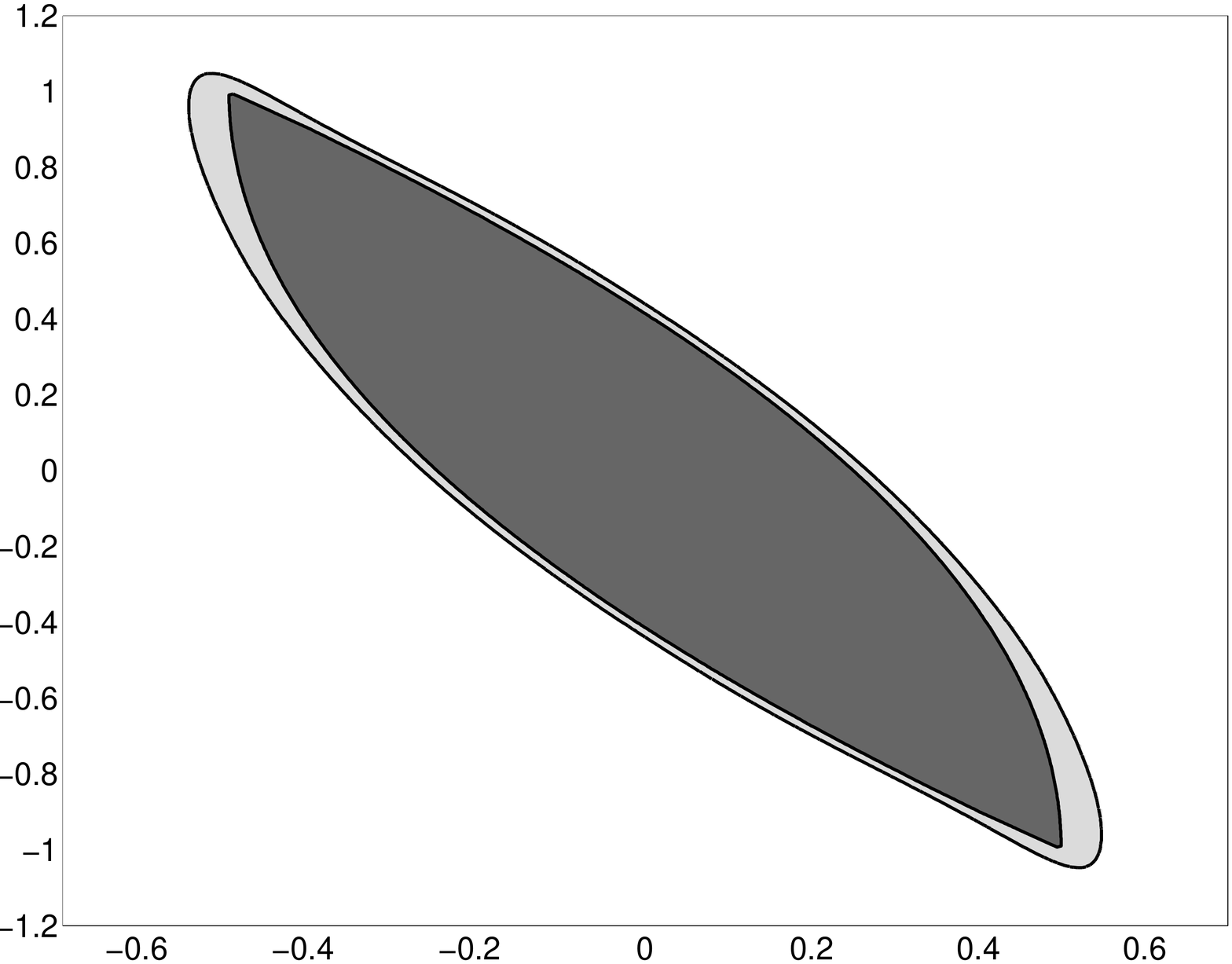}}
\put(70,94){\footnotesize $d = 6$}
\put(200,94){\footnotesize $d = 8$}
\put(328,94){\footnotesize $d = 10$}
\put(455,94){\footnotesize $d = 12$}

\end{picture}
\caption{Double integrator -- semialgebraic outer approximations (light gray) to the ROA (dark gray) for degrees $d\in\{6,8,10,12\}$.}
\label{fig:3}
\end{figure*}


\subsection{Brockett integrator}
Next, we consider the Brockett integrator
\begin{align*}
\dot{x}_1 &= u_1 \\
\dot{x}_2 &= u_2 \\
\dot{x}_3 & =  u_1x_2 - u_2x_1
\end{align*}
with the constraint sets $X = \{x\in \mathbb{R}^3 :  ||x||_\infty \le 1\}$ and $U = \{u\in \mathbb{R}^2 : ||u||_2 \le 1\}$, the target set $X_T = \{0\}$ and the final time $T = 1$. The ROA can be computed analytically (see~\cite{sicon}) as  $X_0 = \{x\in \mathbb{R}^3 \: : \: \mathcal{T}(x)\le 1\}$, where
\[
\mathcal{T}(x) = \frac{\theta\sqrt{x_1^2+x_2^2+2|x_3|}}{\sqrt{\theta+\sin^2\theta-\sin\theta\cos\theta}},
\]
and $\theta = \theta(x)$ is the unique solution in $[0,\pi)$ to
\[
\frac{\theta - \sin\theta\cos\theta}{\sin^2\theta}(x_1^2+x_2^2) = 2|x_3|.
\]
The ROA estimates $X_{0k}$ are not monotonous in this case and therefore in Figure~\ref{fig:Brockett} we rather show the monotonous estimates $\bar{X}_{0k}$ defined in~(\ref{eq:barX0k}) for degrees six $d = 2k \in \{6,10\}$. We observe fairly good tightness of the estimates.
 
\begin{figure*}[!h]
\begin{picture}(140,100)

\put(0,-45){\includegraphics[width=90mm]{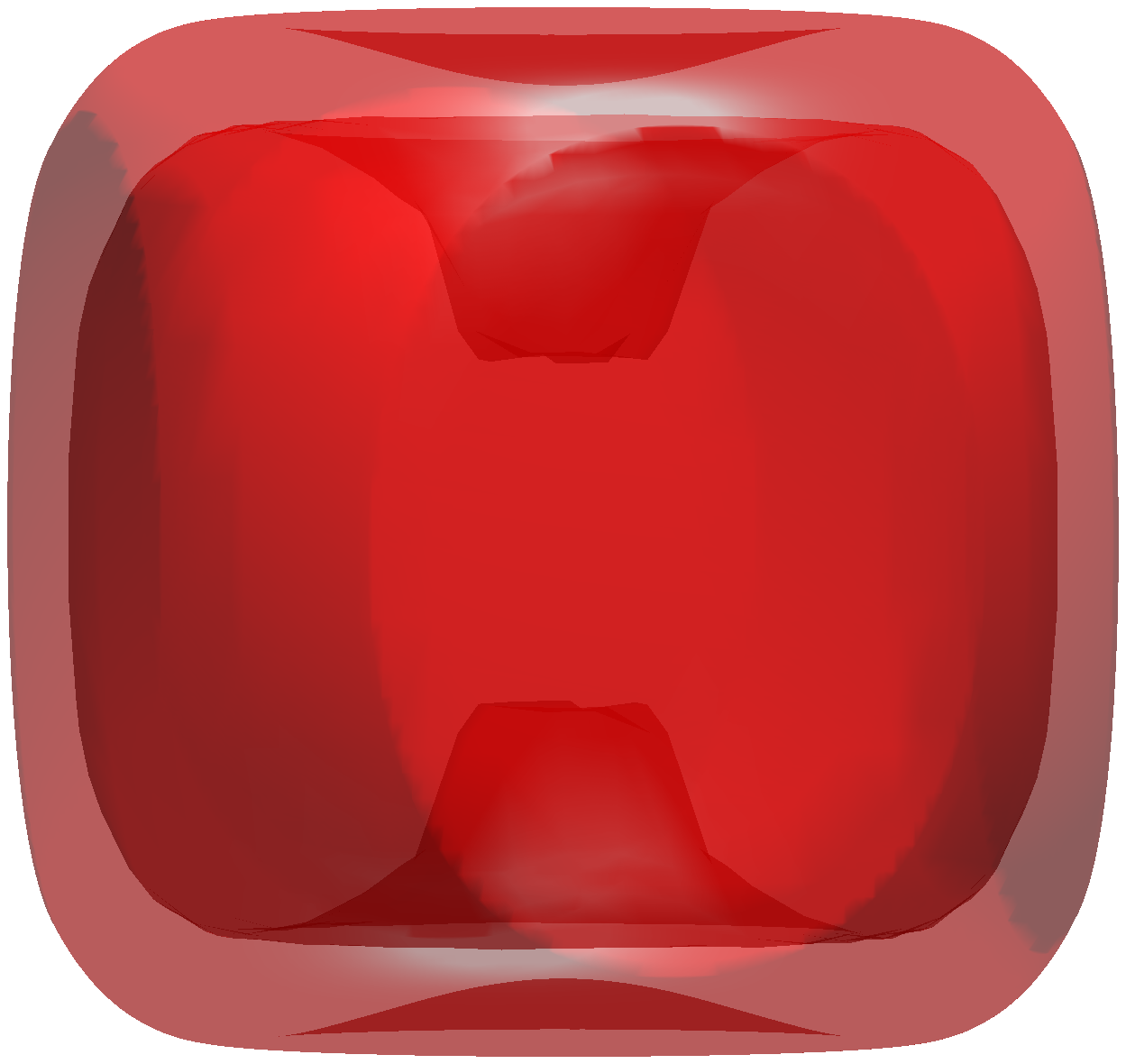}} 
\put(210,-45){\includegraphics[width=90mm]{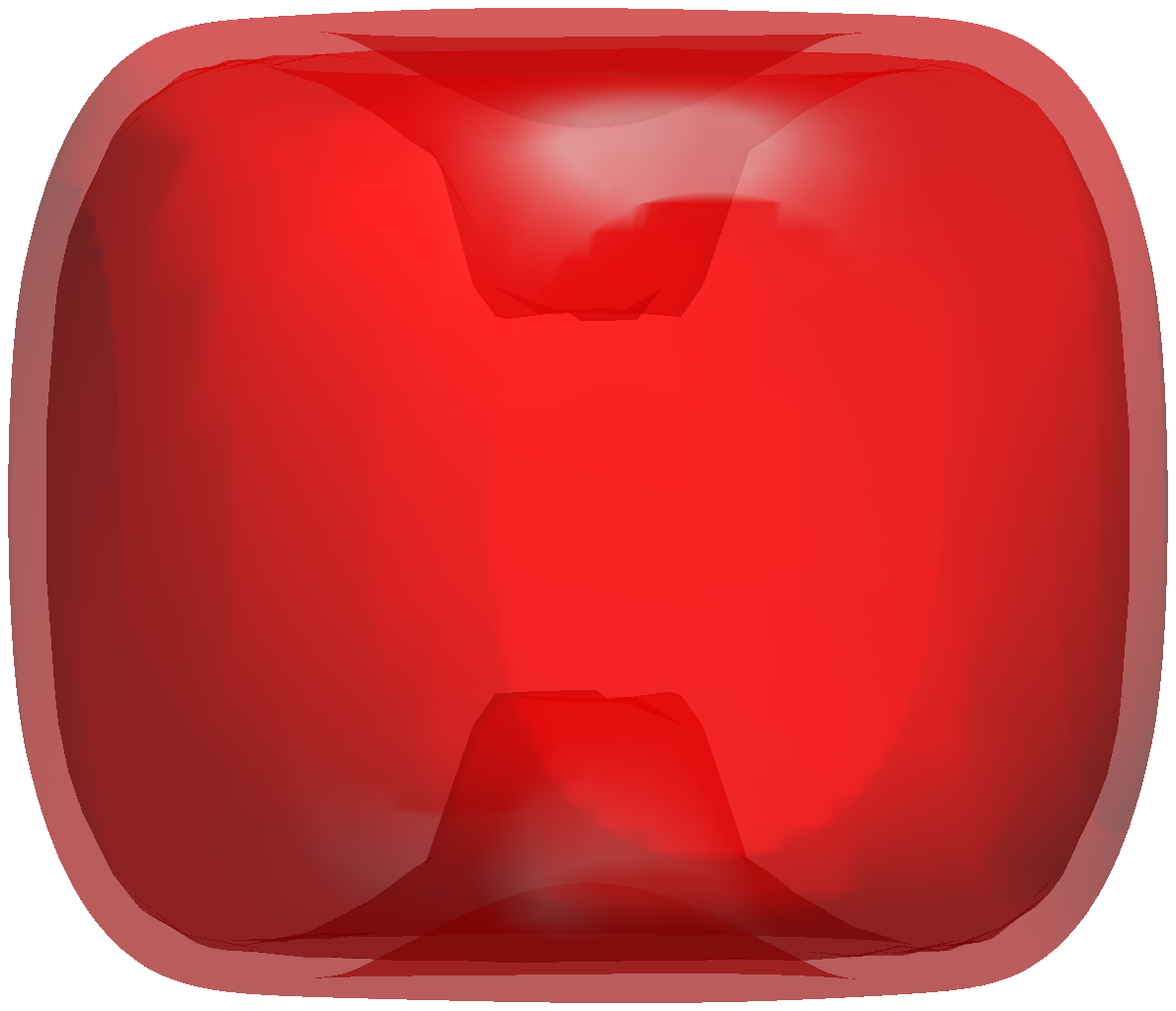}}

\put(123,120){\footnotesize $d = 6$}
\put(330,120){\footnotesize $d=10$}

\end{picture}
\caption{Brockett integrator -- semialgebraic outer approximations (light red, larger) to the ROA (dark red, smaller) for degrees $d\in\{6,10\}$.}
\label{fig:Brockett}
\end{figure*}

\subsection{Acrobot}
As our last example we consider the acrobot system adapted from~\cite{murrayAcrobot}, which is essentially a double pendulum with both joints actuated; see Figure~\ref{fig:acrobotSketch}. The system equations are given by
\[
\dot{x} = \begin{bmatrix}x_3\\x_4\\M(x)^{-1}N(x,u)\end{bmatrix}\in\mathbb{R}^4,
\]
where
\[
M(x) = \footnotesize \begin{bmatrix}			 
			 3+\cos (x_2)     &    1+\cos (x_2) \\
 			1+\cos (x_2)      &    1\\
	    \end{bmatrix}
	    \]
	    and
	    \begin{align*}
	    N(x,u) &= \\
	    &\hspace{-0.9cm}\footnotesize\begin{bmatrix}
                     g\sin(x_1 + x_2) - a_1 x_3 + a_2\sin (x_1) + x_4\sin(x_2)(2x_3 + x_4) + u_1\\
		-\sin(x_2)x_3^2 - a_1 x_4 + g\sin(x_1 + x_2) + u_2
	\end{bmatrix}
\end{align*}
with $g = 9.8$, $a_1 = 0.1$ and $a_2 = 19.6$. The first two states are the joint angles (in radians) and the second two the corresponding angular velocities (in radians per second). The two control inputs are the torques in the two joints. Here, rather than comparing our approximations with the true ROA (which is not easily available), we study how the size of the ROA approximations is influenced by the actuation of the first joint. We consider two cases: with both joints actuated and with only the middle joint actuated. In the first case the input constraint set is $U = [-10,10]\times [-10,10]$ and in the second case it is $U = \{0\}\times [-10,10]$. The state constraint set is for both cases $X = [-\pi/2, \pi/2] \times [-\pi, \pi] \times [-5, 5] \times [-5, 5]$. Since this system is not polynomial we take a third order Taylor expansion of the vector field around the origin. An exact treatment would be possible via a coordinate transformation leading to rational dynamics to which our approach can be readily extended; this extension is, however, not treated in this paper and therefore we use the simpler (and non-exact) approach with Taylor expansion. Figure~\ref{fig:acrobot} shows the approximations $X_{0k}$ of degree $d = 2k \in\{6,8\}$; as expected disabling actuation of the first joint leads to a smaller ROA approximation. For this largest example presented in the paper we also report computation times for two SDP solvers: the recently released MOSEK SDP solver and SeDuMi. Computation times\footnote{Table~\ref{tab:acrobotTime} reports pure solver times, excluding the Yalmip parsing and preprocessing overhead, using Apple iMac with 3.4 GHz Intel Core i7, 8 GB RAM, Mac OS X
10.8.3 and Matlab 2012a.} reported in Table~\ref{tab:acrobotTime} show that MOSEK outperforms SeDuMi in terms of speed by a large margin; this finding does not seem to be specific to this particular problem and holds for all ROA computation problems presented. Before solving, the problem data was scaled such that the constraint sets become unit boxes.

 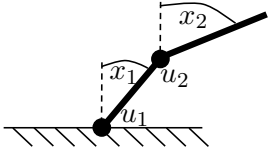
\begin{figure}[h]
	\begin{picture}(80,31)
	\put(195,0){\includegraphics[width=35mm]{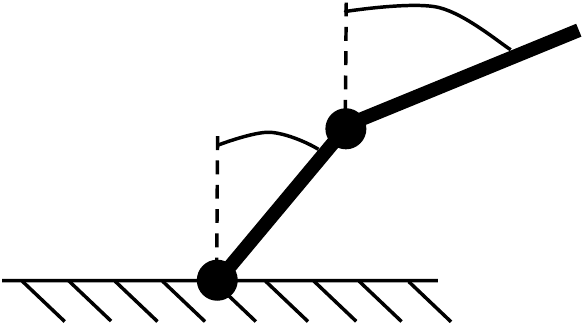}} 
	\put(235,26){\small $x_1$}
	\put(261,47){\small $x_2$}
	\put(239,10){\small $u_1$}
	\put(254,25){\small $u_2$}
	\end{picture}
	\caption{Acrobot -- sketch}
	\label{fig:acrobotSketch}
\end{figure}


\begin{figure*}[!t]
	\begin{picture}(140,220)
	\put(-35,-10){\includegraphics[width=100mm]{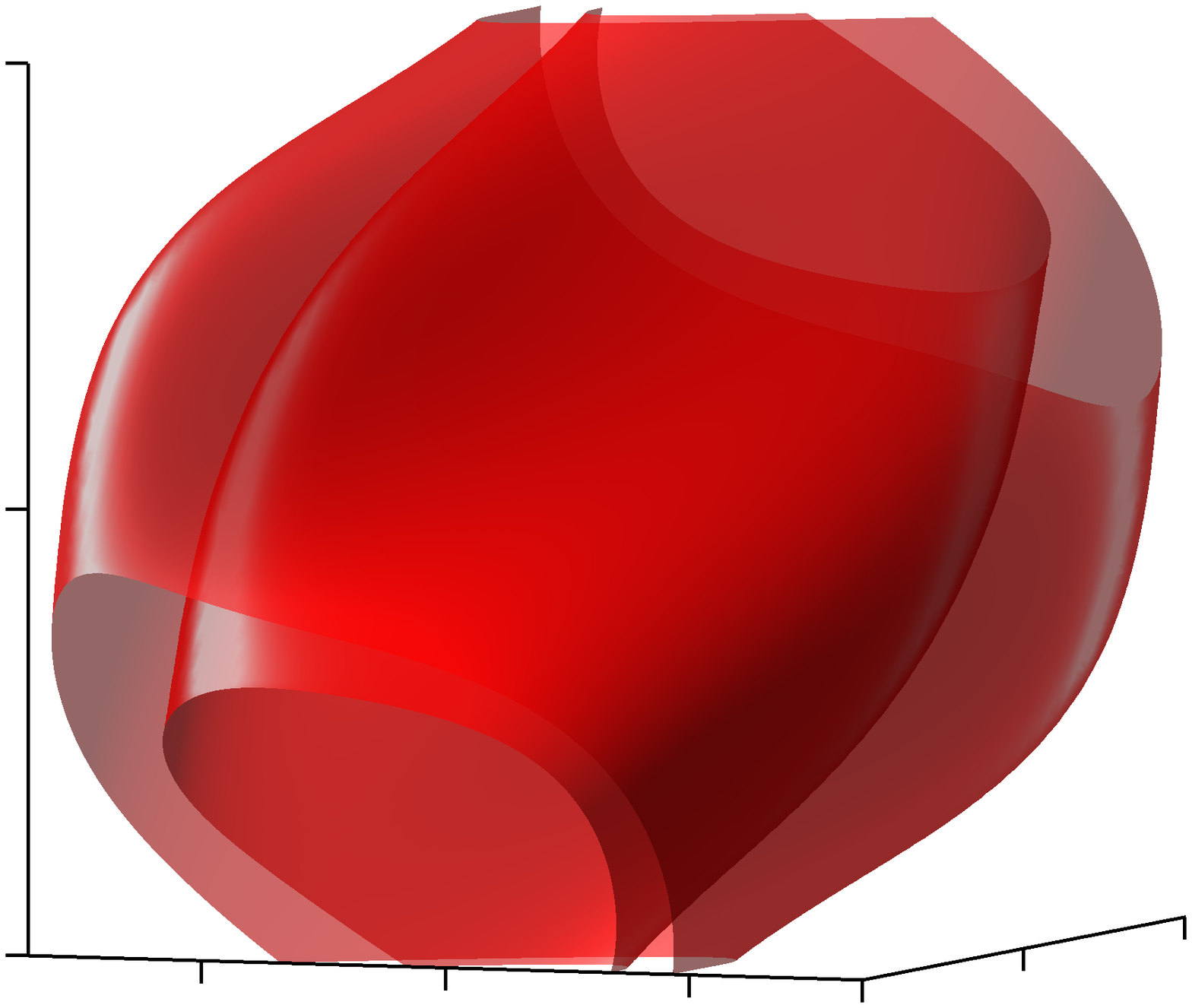}} 
	\put(220,-10){\includegraphics[width=100mm]{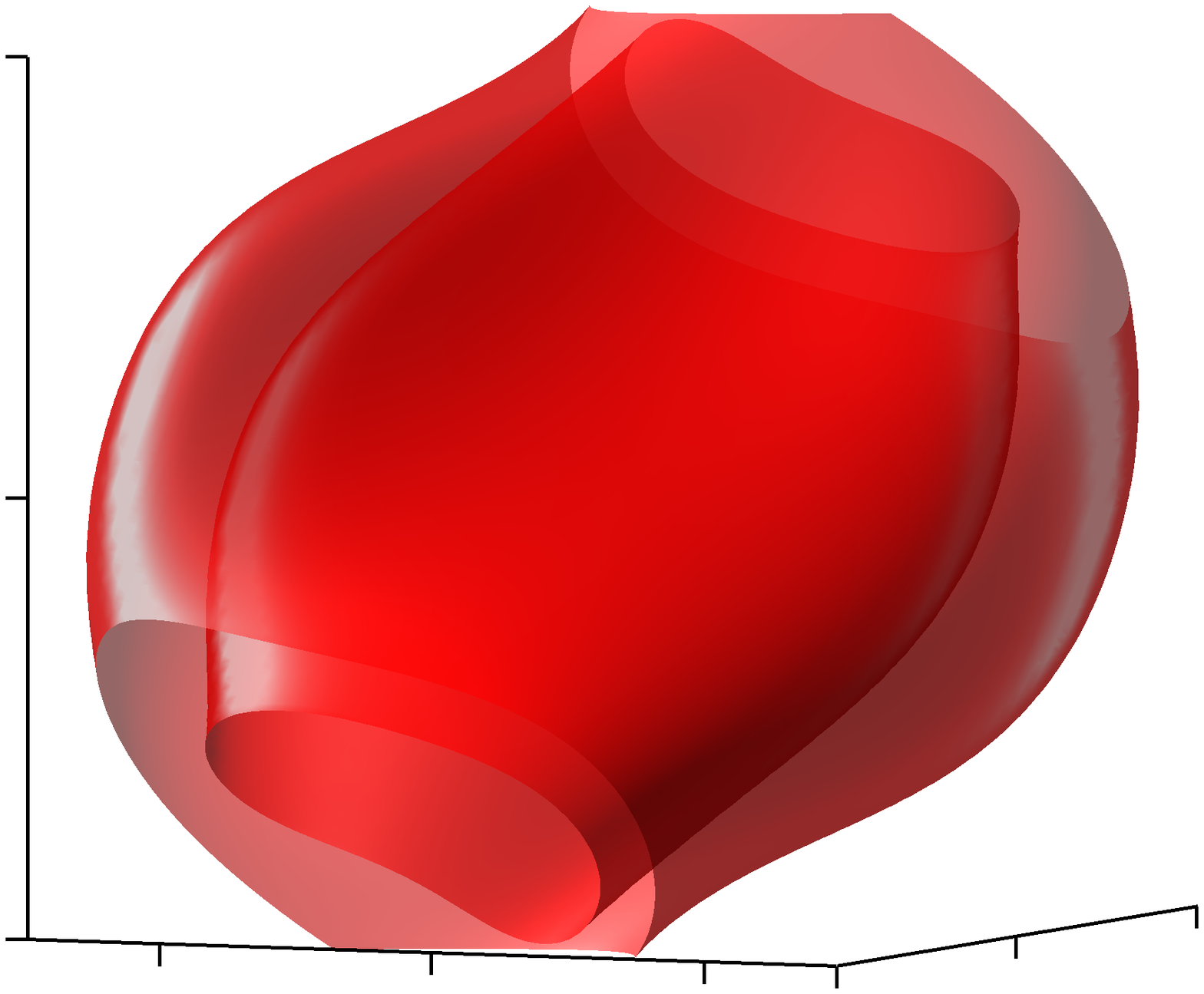}} 

	\put(185,20){\small $x_1$}
	\put(84,10){\small $x_2$}
	\put(-8,99){\small $x_3 $}
	
	\put(440,19){\small $x_1$}
	\put(337,10){\small $x_2$}
	\put(244,99){\small $x_3 $}
	
	\put(41,164){\small $d = 6$}
	\put(301,168){\small $d = 8$}
	
	\put(40,22){\footnotesize $-\pi/2$}
	\put(85,21.5){\footnotesize $0$}
	\put(115,21){\footnotesize $\pi/2$}
	
	\put(290,24){\footnotesize $-\pi/2$}
	\put(339,21.5){\footnotesize $0$}
	\put(373,21){\footnotesize $\pi/2$}
	
	\put(151,26){\footnotesize $-\pi/2$}
	\put(177,29){\footnotesize $0$}
	\put(188,33){\footnotesize $\pi/2$}
	
	\put(401,26){\footnotesize $-\pi/2$}
	\put(427,29){\footnotesize $0$}
	\put(439,32){\footnotesize $\pi/2$}
	
	\put(8,35){\footnotesize $-5$}
	\put(11,98){\footnotesize $0$}
	\put(11,162){\footnotesize $5$}
	
	\put(262,35){\footnotesize $-5$}
	\put(265,98){\footnotesize $0$}
	\put(265,162){\footnotesize $5$}

	\end{picture}
	\caption{Acrobot -- section for $x_4 = 0$ of the semialgebraic outer approximations of degree $d\in\{6,8\}$. Only the middle joint actuated -- darker, smaller; both joints actuated -- lighter, larger. The states displayed $x_1$, $x_2$ and $x_3$ are, respectively, the lower pendulum angle, the upper pendulum angle and the lower pendulum angular velocity. }
	\label{fig:acrobot}
\end{figure*}

\begin{table}[h]
\centering
\caption{\rm Acrobot -- comparison of computation time of MOSEK and SeDuMi for different degrees of the approximating polynomial. The ``--'' in the last cell signifies that SeDuMi could not solve the problem.}\label{tab:acrobotTime}\vspace{2mm}
\begin{tabular}{cccc}
\toprule
degree & 4 & 6 & 8 \\\midrule
MOSEK & 0.93\,s & 23.5\,s & 2029\,s \\\midrule
SeDuMi & 7.1\,s & 2775\,s & -- \\
\bottomrule
\end{tabular}
\end{table}

\section{Conclusion}\label{sec:Conclusion}

The main contributions of this paper can be summarized as follows:
\begin{itemize}
\item contrary to most of the existing systems control literature, we propose a convex
formulation for the problem of computing the controlled region of attraction;
\item our approach is constructive in the sense that we rely on standard
hierarchies of finite-dimensional LMI relaxations whose convergence can be
guaranteed theoretically and for which public-domain interfaces and
solvers are available;
\item we deal with polynomial dynamics and semialgebraic input and state constraints,
therefore covering a broad class of nonlinear control systems;
\item additional properties (e.g., convexity) of the approximations can be enforced by additional constraints on the approximating polynomial $v(0,\cdot)$ (e.g., Hessian being negative definite).
\item the approach is extremely simple to use -- the outer approximations are the outcome of a single semidefinite program with \emph{no additional data} required besides the problem description. 
\end{itemize}

The problem of computing the (forward) reachable set, i.e. the set of all states that can be reached from a given set of initial conditions under input and state constraints, can be addressed with the same techniques -- the problem can be formulated as ROA computation using a time-reversal argument. Similar ideas can also be used to characterize and compute outer approximations of the maximum controlled invariant set (both for discrete and continuous time); this is a work in progress.

The hierarchy of LMI relaxations described in this paper generates a sequence of nested outer approximations of the ROA, but it is also possible, using a similar approach, to compute valid
inner approximations. Results for uncontrolled systems will be reported elsewhere. The extension to the controlled case is far more involved (both theoretically and computationally) and is a subject of future research.

Furthermore, there is a straightforward extension to piecewise polynomial dynamics defined over a semialgebraic partition of the state and input spaces -- one measure is then defined for each region of the partition. Our approach should also allow for extensions to stochastic systems (either discrete-time controlled Markov processes or controlled SDEs) and/or uncertain systems.

Since it is based on (an extension of) Lasserre's hierarchy of LMI relaxations (originally proposed
for static polynomial optimization \cite{lasserre}), our approach scales
similarly as Lasserre's approach. Namely, the number of
moments (variables in the LMI relaxation) grows as $O(k^{n+m})$ when the problem dimension $n+m$ is kept constant and the relaxation order $k = d/2$ varies, and grows as $O((n+m)^k)$ when $k$ is kept constant and $n+m$ varies. Therefore, at present, the approach is limited to systems of moderate size (say, $n+m \le 6$) unless one is willing to compromise the accuracy of the approximations by taking a small relaxation order $k$. However, given the rapid progress of computing and optimization, the authors expect the approach to scale to larger dimensions in the future. One possible direction is sparsity exploitation; indeed, the recently released MOSEK SDP solver seems to have far superior performance to SeDuMi on our problem class, most likely due to more sophisticated sparsity exploitation. Another direction is parallelization; for instance, parallel interior point solvers (e.g., SDPARA~\cite{sdpara}) should allow the approach to scale to larger dimensions. Alternative algorithms to standard primal-dual interior-point methods, such as first-order methods, are also worth considering. This is currently investigated by the authors.


Numerical examples indicate that the choice of monomials as a dense basis for
the set of continuous functions on compact sets, while mathematically appropriate
(and notationally convenient), is not always satisfactory in terms of convergence
and quality of the approximations due to numerical ill-conditioning of this basis. However, this is not
peculiar to ROA computation problems -- a similar behavior was already observed
when computing the volume (and moments) of semialgebraic sets in~\cite{volume}.
To achieve better performance, we recommend the use of alternative polynomial bases
such as Chebyshev polynomials; see~\cite[Section~4]{henrionKybernetika} for more details.

\section*{Appendix~A}

In this Appendix we state and prove the correspondence between the Liouville PDE
on measures (\ref{eq:Liouville}) and the convexified differential inclusion (\ref{eq:incl2}). We will need the notion of a stochastic kernel. Let $Y$ and $Z$ be two Borel sets (in two Euclidean spaces of not necessarily the same dimension). The object $\nu(\cdot \!\mid\! \cdot)$ is called a stochastic kernel on $Y$ given $Z$ if, first, $\nu(\cdot \!\mid\! z)$ is a probability measure on $Y$ for every fixed $z\in Z$, and, second, if $\nu(A \!\mid\!\cdot)$ is a measurable function on $Z$ for every fixed $A\subset Y$. Let also $\bar{\mu}(t,x)$ denote the $(t,x)$-marginal of the occupation measure $\mu$ defined through~(\ref{eq:ocmeas}), that is,
\[
\bar{\mu}(A\times B) := \mu(A\times B\times U)\quad \forall\: A \subset [0,T],\: B\subset X.
\]

\begin{lemma}\label{lem:corr}
\looseness-1
Let $(\mu_0,\mu,\mu_T)$ be a triplet of measures satisfying the Liouville equation~(\ref{eq:Liouville}) such that $\mathrm{spt}\,\mu_0\subset X$, $\mathrm{spt}\,\mu\subset [0,T]\times X \times U$ and $\mathrm{spt}\,\mu_T \subset X_T$. Then there exists a family of absolutely continuous admissible trajectories of~(\ref{eq:incl2}) starting from $\mu_0$ (i.e., trajectories in $\bar{\mathcal{X}}(x_0)$) such that the occupation measure and the terminal measure generated by this family of trajectories are equal to $\bar{\mu}$ and $\mu_T$, respectively.
\end{lemma}
\begin{proof}
Since the occupation measure $\mu$ is defined on a Euclidean space which is Polish and therefore Souslin, it can be, in view of \cite[Corollary 10.4.13]{bogachev}, disintegrated as \[d\mu(t,x,u) = d\nu(u\!\mid\! t,x)d\bar{\mu}(t,x),\] where $d\nu(u\!\mid\! t,x)$ is a stochastic kernel on $U$ given $[0,T]\times X$. Then we can rewrite equation~(\ref{test}) as
\begin{align}\label{eq:testApp}
\nonumber & \int_{X_T} v(T,\cdot)\,d\mu_T -  \int_X v(0,\cdot)\,d\mu_0\\ \nonumber &= \int_{[0,T]\times X}\int_U \frac{\partial v}{\partial t} + \mathrm{grad}\,v\cdot f(t,x,u)\,d\nu(u\!\mid\! t,x)\,d\bar{\mu}(t,x)\\ \nonumber
&=\int_{[0,T]\times X}  \frac{\partial v}{\partial t} + \mathrm{grad}\,v\cdot\Big[\int_U f(t,x,u)\,d\nu(u\!\mid\! t,x)\Big]\,d\bar{\mu}(t,x)\\
&=  \int_{[0,T]\times X}  \frac{\partial v}{\partial t} + \mathrm{grad}\,v\cdot\bar{f}(t,x)\,d\bar{\mu}(t,x),
\end{align}
where
\[\bar{f}(t,x) := \int_U f(t,x,u)\,d\nu(u\!\mid\! t,x) \in \mathrm{conv}\, f(t,x,U).\]
Therefore we will study the trajectories of the differential equation
\begin{equation}\label{eq:oderel}
	\dot{x}(t) = \bar{f}(t,x(t)).
\end{equation}
In the remainder of the proof we show that the measures $\mu_T$ and $\bar{\mu}$ are generated by a family of absolutely continuous trajectories of this differential equation (which is clearly a subset of trajectories of the convexified inclusion~(\ref{eq:incl2})) starting from $\mu_0$. Note that the vector field $\bar{f}$ is only known to be measurable\footnote{Measurability of $\bar{f}(t,x)$ follows by first observing that for $f(t,x,u) =  I_{A\times B\times C}(t,x,u)$ we have $\bar{f}(t,x) = I_{A}(t)I_B(x)\nu(C\!\mid\! t,x)$, which is a product of measurable functions, and then by approximating an arbitrary measurable $f(t,x,u)$ by simple functions (i.e., sums of indicator functions). This is a standard measure theoretic argument; details are omitted for brevity.}, so this equation may not admit a unique solution.

Observe that the $t$-marginal of $\mu$ (and hence of $\bar{\mu}$) is equal to the Lebesgue measure restricted to $[0,T]$ scaled by $\rho:=\mu_0(X)$ (=$\mu_T(X)$). Indeed, plugging $v(t,x) = t^k$, $k\in\mathbb{N}$, in (\ref{test}), we obtain $\mu_T(X) = \int t^k d\mu_0 + \int k t^{k-1}\,d\mu$; taking $k = 0$ gives $\mu_T(X) = \mu_0(X)$ and $k\ge 1$ gives $\int t^{k-1}\,d\mu = \mu_T(X) T^{k}/k$, which is nothing but the Lebesgue moments on $[0,T]$ scaled by $\mu_T(X) = \mu_0(X)$. Therefore, using~\cite[Theorem~6.4]{bogachev}, we can disintegrate $\bar{\mu}$ as
\begin{equation}\label{eq:mudisint}
	d\bar{\mu}(t,x) = d\mu_t(x)dt,
\end{equation}
where $d\mu_t(x)$ is a stochastic kernel on $X$ given $t$ scaled by $\rho$ and $dt$ is the standard Lebesgue measure on $[0,T]$. The kernel $\mu_t$ can be thought of as the distribution\footnote{It will become clear from the following discussion that for $t=0$ and $t=T$ this kernel (or a version thereof) coincides with $\mu_0$ and $\mu_T$, respectively; hence there is no ambiguity in notation. Note also that the kernel $\mu_t$, $t\in [0,T]$, is defined uniquely up to a subset of $[0,T]$ of Lebesgue measure zero; by a ``version'' we then mean a particular choice of the kernel.} of the state at time $t$. The kernel $\mu_t$ is defined uniquely $dt$-almost everywhere, and we will show that there is a version such that the function $t\mapsto \int_{X}w(x)\,d\mu_t(x)$ is absolutely continuous for all $w\in C^1(X)$ and such that the continuity equation
\begin{equation}\label{eq:cont}
	\frac{d}{dt}\!\int_X w(x)\,d\mu_t(x) = \int_X \mathrm{grad}\,w(x)\cdot\bar{f}(t,x)\,d\mu_t(x)\hspace{0.4em}\forall\,w\in C^1(X)
\end{equation}
with the initial condition $\mu_0$ is satisfied almost everywhere w.r.t. the Lebesgue measure on $[0,T]$.


Fix $w\in C^1(X)$ and define the test function $v(t,x):=\psi(t)w(x)$, where $\psi\in C^1([0,T])$. Then from equation~(\ref{eq:testApp})
\begin{align*}
\psi(T)&\int_{X_T} w\,d\mu_T -  \psi(0)\int_X w\,d\mu_0 \\ &= \int_{[0,T]\times X}\hspace{-1em} \frac{\partial (\psi w)}{\partial t} + \mathrm{grad}(\psi w)\cdot\bar{f}(t,x)\,d\bar{\mu}(t,x)\\
&=  \int_0^T\! \int_X \dot{\psi}(t)w(x) + \psi(t)\mathrm{grad}\,w(x)\cdot\bar{f}(t,x)\,d\mu_t(x)dt\\
& = \int_0^T\Big[ \dot{\psi}\int_X w\,d\mu_t +\psi\int_X \mathrm{grad}\,w\cdot\bar{f}\,d\mu_t\Big] dt,
\end{align*}
which can be seen as an equation of the form
\begin{equation}\label{eq:lemmaAux}
	\psi(T)d\,-\,\psi(0)c=\! \int_0^T\!\!\dot{\psi}(t)a(t) + \psi(t)b(t)\,dt \hspace{0.5em} \forall \psi\in C^1([0,T]),
\end{equation}
where $c := \int_X w(x)\,d\mu_0(x)$, $d:=\int_{X_T} w(x)\,d\mu_T$ and $b(t) := \int_X \mathrm{grad}\,w\cdot\bar{f}(t,x)\,d\mu_t(x)$ are constants and $a(t)$ is an unknown function. One solution is clearly $a(t) = \int_X w\,d\mu_t$. Now we show that
\[
\tilde{a}(t) := c + \int_0^t b(\tau)\,d\tau = \int_X w\,d\mu_0 + \int_{0}^t\int_X \mathrm{grad}\,w\cdot \bar{f}\,d\mu_\tau d\tau
\]
also solves the equation. Indeed, since from~(\ref{eq:testApp}) with $v$ replaced by $w$ we have $\tilde{a}(T) = \int_X w\,d\mu_T = d$, integration by parts gives
\[\int_0^T \dot{\psi}(t)\tilde{a}(t)\,dt = \psi(T)d-\psi(0)c - \int_0^T \psi(t)b(t)\,dt,  \]
so $\tilde{a}(t)$ indeed solves equation~(\ref{eq:lemmaAux}). Now we prove that this solution is unique. Since $\tilde{a}$ is a solution we have
\[
\psi(T)d - \psi(0)c = \int_0^T \dot{\psi}(t)\tilde{a}(t) + \psi(t)b(t)\,dt,
\]
and subtracting this from~(\ref{eq:lemmaAux}) we get
\[
	0 = \int_0^T \dot{\psi}(t)[a(t) - \tilde{a}(t)]\,dt\quad \forall\: \psi\in C^1([0,T]),
\]
or equivalently
\[
0 = \int_0^T \phi(t)[a(t) - \tilde{a}(t)]\,dt\quad \forall\: \phi\in C([0,T]).
\]
Since $C([0,T])$ is dense in $L^1([0,T])$, this implies $a(t) = \tilde{a}(t)$ $dt$-almost everywhere. Consequently, since $C^1(X)$ is separable,
\begin{equation}\label{eq:appaAux}
\int_X \! w\,d\mu_t \!=\!\! \int_X w(x)d\mu_0 +\! \int_0^t\!\! \int_X\!\! \mathrm{grad}\,w\cdot\bar{f}\,d\mu_\tau d\tau\hspace{0.5em} \forall\: w \in C^1(X)
\end{equation}
$dt$-almost everywhere. The right-hand side of this equality is an absolutely continuous function of time for each $w\in C^1(X)$ and the left-hand side is a bounded positive linear functional on $C(X)$ for all $t\in [0,T]$. By continuity of the right-hand-side of~(\ref{eq:appaAux}) with respect to time, this right-hand side is a bounded positive linear functional on $C^1(X)$ for all $t\in[0,T]$ and therefore can be uniquely extended to a bounded positive linear functional on $C(X)$ (since $C^1$ is dense in $C$). Therefore, for all $t\in [0,T]$ the right-hand side has a representing measure~\cite[Theorem 2.14]{rudin} and hence there is a version of $\mu_t$ such that the equality~(\ref{eq:appaAux}) holds for \emph{all} $t\in [0,T]$. With this version of $\mu_t$ the function $t\mapsto \int_X w(x)\,d\mu_t(x)$ is absolutely continuous and $\mu_t$ solves the continuity equation~(\ref{eq:cont}).

To finish the proof, we use \cite[Theorem~3.2]{ambrosio} which asserts the existence of a nonnegative measure $\sigma$ on $C([0,T];\mathbb{R}^n)$ which corresponds to a family of absolutely continuous solutions to ODE~(\ref{eq:oderel}) whose projection at each time $t\in[0,T]$ coincides with $\mu_t$. More precisely, there is a nonnegative measure $\sigma\in M(C([0,T];\mathbb{R}^n))$ supported on a family of absolutely continuous solutions to ODE~(\ref{eq:oderel}) such that for all measurable $w:\mathbb{R}^n\to \mathbb{R}$
\begin{equation}\label{eq:contmeas}
\int_X w(x)\mu_t(x) = \int_{C([0,T];\,\mathbb{R}^n)}\hspace{-1.5em}w(x(t))\,d\sigma(x(\cdot))\quad \forall\: t\in [0,T].
\end{equation}
Using $I_{A\times B}(t,x) = I_A(t)I_B(x)$, it follows from~(\ref{eq:mudisint}) that
\[\bar{\mu}(A\times B)=\int_{[0,T]\times X}I_A(t)I_B(x)\,d\bar{\mu}(t,x)=\int_0^T I_A(t)\int_X I_B(x) \,d\mu_t(x)\, dt.\]
Therefore, using (\ref{eq:contmeas}) with $w(x)=I_B(x)$ and Fubini's theorem~\cite[Theorem 8.8]{rudin}, we get
\[\bar{\mu}(A\times B) = \int_{C([0,T];\,\mathbb{R}^n)}\int_0^T I_{A\times B}(t,x(t))\,dt\, d\sigma(x(\cdot)),\]
and so the occupation measure of the family of trajectories coincides with $\bar{\mu}$. Clearly, the initial and the final measures of this family coincide with $\mu_0$ and $\mu_T$ as well. As a result $\sigma$-almost all trajectories of this family are admissible. The proof is completed by discarding the null-set of trajectories that are not admissible, which does not change the measure $\sigma$ and the generated measures $\bar{\mu}$, $\mu_0$, $\mu_T$.
\end{proof}

\section*{Appendix~B}

In this Appendix we elaborate further on the discussion from Section~\ref{sec:relaxed} on the connection between the classical ROA and the relaxed ROA. Let us recall the definition of the classical ROA
\[
X_0 := \big\{x_0\in X \: : \: \mathcal{X}(x_0)\neq \emptyset \big\},
\]
where
\begin{align*}
\mathcal{X}(x_0)  :=  \big\{x(\cdot)\::\:\dot{x}(t)\in f(t,x(t),U)\:\text{a.e.},\: x(0)=x_0,
x(T)\in X_T,\: x(t)\in X\: \forall t\in [0,T]\big\}
\end{align*}
and $x(\cdot)$ is required to be absolutely continuous. Similarly, recall the definition of the relaxed~ROA
\[
\bar{X}_0 := \big\{x_0\in X \: : \: \bar{\mathcal{X}}(x_0)\neq \emptyset \big\},
\]
where
\begin{align*}
\bar{\mathcal X}(x_0)  :=  \big\{x(\cdot)\::\:&\dot{x}(t)\in \mathrm{conv}\:f(t,x(t),U)\:\text{a.e.},\: x(0)=x_0,\\
&x(T)\in X_T,\: x(t)\in X\: \forall t\in [0,T]\big\}
\end{align*}
with $x(\cdot)$ absolutely continuous. Obviously, it holds
\begin{equation}\label{inclusion}
X_0 \subset \bar{X}_0,
\end{equation}
and the question is whether this inclusion is strict or not.

Denote $B_\epsilon(a) := \{x\in\mathbb{R}^n \: : \: ||x-a||_2 < \epsilon \}$ and define the dilated constraint sets
\[  X^\epsilon := X \oplus B_\epsilon(0) \quad \mathrm{and}\quad   X_T^\epsilon := X_T \oplus B_\epsilon(0), \]
where $\oplus$ denotes the Minkowski sum of two sets. Accordingly, the dilated ROA and the dilated relaxed ROA are
\[
\begin{array}{rcl}
X^{\epsilon}_0 & := & \big\{x_0\in X \: : \: \mathcal{X}^{\epsilon}(x_0)\neq \emptyset \big\},\\
\bar{X}^{\epsilon}_0 & := & \big\{x_0\in X \: : \: \bar{\mathcal{X}}^{\epsilon}(x_0)\neq \emptyset \big\},
\end{array}
\]
where
\begin{align*}
\mathcal{X}^{\epsilon}(x_0)  &:=  \big\{x(\cdot)\::\:\dot{x}(t)\in f(t,x(t),U)\:\text{a.e.},\: x(0)=x_0,
x(T)\in X_T^\epsilon,\: x(t)\in X^\epsilon\: \forall t\in [0,T]\big\},\\
\bar{\mathcal X}^\epsilon(x_0)  &:=  \big\{x(\cdot)\::\:\dot{x}(t)\in \mathrm{conv}\:f(t,x(t),U)\:\text{a.e.},\: x(0)=x_0,
x(T)\in X_T^\epsilon,\: x(t)\in X^\epsilon\: \forall t\in [0,T]\big\}.
\end{align*}

Since the constraint sets are compact and the vector field $f$ Lipschitz, it follows from the equivalence between the trajectories of the convexified inclusion~(\ref{eq:incl2}) and solutions to the Liouville equation~(\ref{eq:Liouville}), stated in Lemma~\ref{lem:corr} of Appendix~A, and from Filippov-Wa$\dot{\mathrm z}$ewski's relaxation Theorem (see, e.g., \cite{aubin})
that
\[
\bar{X}_0 =\mathrm{spt}\,\mu_0 \:\subset\: \bigcap_{\epsilon > 0} X_0^\epsilon.
\]
In contrast, for all $\epsilon>0$ it holds
\[
X^{\epsilon}_0 = \bar{X}^{\epsilon}_0.
\]

In general inclusion (\ref{inclusion}) is strict. However, we argue that for most practical purposes the relaxed ROA $\bar{X}_0$ and the true ROA $X_0$ are the same.
Indeed, for any $x_0\in  \bar{X}_0$ there exists a sequence of admissible control functions $u_k(\cdot)$ such that
\[
	\sup_{t\in [0,T]} \mathrm{dist}_{X}(x_k(t)) \to 0\quad \mathrm{and}\quad \mathrm{dist}_{X_T}(x_k(T)) \to 0
\]
as $k\to\infty$, where $x_k(\cdot)$ denotes the solution to the ODE~(\ref{sys}) corresponding to the control function $u_k(\cdot)$, and $\mathrm{dist}_A(x):=\inf\{\|z-x\|_2\: : \: z\in A \}$ denotes the distance to a set $A$.

\section*{Appendix~C}

In this Appendix we describe two contrived examples of control systems (\ref{sys}) for which the relaxed ROA $\bar{X}_0$
is strictly larger than the classical ROA $X_0$; see Appendix~B for definitions.

Let $f(t,x,u)=u$, $U=\{-1,+1\}$,
$X=X_T=\{0\}$ for, e.g., $T=1$. Obviously there is no admissible trajectory
in ${\mathcal X}(0)$, whereas there is a feasible triplet of measures satisfying~(\ref{eq:Liouville}) given by $\mu_0 = \delta_0$, $\mu_T = \delta_0$ and
$\mu = \lambda_{[0,1]} \otimes\delta_0\otimes \frac{1}{2}(\delta_{-1}+ \delta_{+1})$, where $\lambda_{[0,1]}$ denotes the restriction of the Lebesgue measure to $[0,1]$. Therefore in this case $X_0=\emptyset \neq \bar{X}_0 = \{0\}$, but $\lambda(X_0)=\lambda(\bar{X}_0)$. Assumption~\ref{relaxed} is therefore satisfied. Note that the relaxed solution corresponds to an infinitely fast chattering of the control input between $-1$ and $+1$ which can be arbitrarily closely approximated by chattering solutions of finite speed; the singleton constraint set $X$, however, renders such solutions infeasible.

Another example for which the gap (e.g., in volume)
between $\bar{X}_0$ and $X_0$ can be as large as desired is the following. Consider $\dot{x} = u\in\mathbb{R}^2$ with
\[x \in X:= B_r(c_1) \cup ([-1,1]\times\{0\}) \cup B_1(c_2) \subset {\mathbb R}^2\] 
with the centers $c_1 = (-1-r,0)$ and $c_2 = (+2,0)$ and a given radius $r>0$. The input and terminal constraints are $u \in U := \{-1,1\}^2$ and $X_T := B_1(c_2)$. That is, the constraint set consist of two balls (one of radius $r$ and the other of radius~$1$) connected by a line; the target set is the ball of radius~1. Then $X_0 = X_T$ is strictly smaller
than $\bar{X}_0 = X$, and $\lambda(X_0) = \pi$, whereas $\lambda(\bar{X}_0) = (1+r^2)\pi$. Assumption~\ref{relaxed} is therefore not satisfied for $r>0$. In this example, regular solutions starting in the left ball cannot transverse the line to the right ball; this is, by contrast, possible for the relaxed solutions using an infinitely fast chattering.

\section*{Appendix~D}
In this appendix we prove Theorem~\ref{lem:noGapRelax}. In order to prove the theorem we rewrite primal LMI problem (\ref{plmi}) in a vectorized form as follows
\begin{equation}\label{primal}
\begin{array}{rcll}
p_k^* & = & \min & \mathbf{c}' \mathbf{y} \\
& & \mathrm{s.t.} & \mathbf{A}\mathbf{y} = \mathbf{b} \\
& & & \mathbf{e}+\mathbf{D}\mathbf{y} \in \mathbf{K},
\end{array}
\end{equation}
where $\mathbf{y}:=[y',\:y'_0,\:y'_T,\hat{y}'_0]'$ and $\mathbf{K}$ is a direct product of cones of positive
semidefinite matrices of appropriate dimensions, here corresponding to the moment matrix
and localizing matrix constraints. The notation $\mathbf{e}+\mathbf{D}\mathbf{y} \in
\mathbf{K}$ means that vector $\mathbf{e}+\mathbf{D}\mathbf{y}$ contains entries of positive
semidefinite moment and localizing matrices, and by construction matrix $\mathbf{D}$ has full
column rank (since a moment matrix is zero if and only if the corresponding moment vector is zero). Dual LMI problem (\ref{dlmi}) then becomes
\begin{equation}\label{dual}
\begin{array}{rcll}
d_k^* & = & \max & \mathbf{b}'\mathbf{x}-\mathbf{e}'\mathbf{z} \\
& & \mathrm{s.t.} & \mathbf{A}'\mathbf{x}+\mathbf{D}'\mathbf{z} = \mathbf{c} \\
& & & \mathbf{z} \in \mathbf{K},
\end{array}
\end{equation}
and we want to prove that $p_k^* = d_k^*$.
The following instrumental result is a minor extension of a classical lemma of the alternatives
for primal LMI (\ref{primal}) and dual LMI (\ref{dual}). The notation $\mathrm{int}\:\mathbf{K}$
stands for the interior of $\mathbf{K}$.

\begin{lemma}\label{lem:alternative}
If matrix $\mathbf{D}$ has full column rank, exactly one of these statements is true:
\begin{itemize}
\item there exists $\mathbf{x}$ and $\mathbf{z}\in \mathrm{int}\:\mathbf{K}$ such that $\mathbf{A}'\mathbf{x}+\mathbf{D}'\mathbf{z}=\mathbf{c}$
\item there exists $\mathbf{y}\neq 0$ such that $\mathbf{A}\mathbf{y}=0$, $\mathbf{D}\mathbf{y} \in \mathbf{K}$
and $\mathbf{c}'\mathbf{y} \leq 0$.
\end{itemize}
\end{lemma}

{\it Proof of Lemma \ref{lem:alternative}:}
A classical lemma of alternatives states that if matrix $\mathbf{\bar{D}}$ has full column rank, then either
there exists $\mathbf{z}\in \mathrm{int}\:\mathbf{K}$ such that $\mathbf{\bar{D}}'\mathbf{z}=\mathbf{\bar{c}}$ or
there exists $\mathbf{\bar{y}}$ such that $\mathbf{\bar{D}}\mathbf{\bar{y}} \in \mathbf{K}$ and
$\mathbf{\bar{c}}'\mathbf{\bar{y}} \leq 0$,
but not both, see e.g. \cite[Lemma 2]{trnovska} for a standard proof based on the geometric
form of the Hahn-Banach separation theorem. Our proof then follows from restricting
this lemma of alternatives to the null-space of matrix $\mathbf{A}$. More explicitly, there exists
$\mathbf{x}$ and $\mathbf{z}$ such that $\mathbf{A}'\mathbf{x}+\mathbf{D}'\mathbf{z}=\mathbf{c}$
if and only if $\mathbf{z}$ is such that $\mathbf{\bar{D}}'\mathbf{z} = \mathbf{\bar{c}}$ with
$\mathbf{\bar{D}} = \mathbf{D}\mathbf{F}$, $\mathbf{\bar{c}}=\mathbf{F}'\mathbf{c}$ for
$\mathbf{F}$ a full-rank matrix such that $\mathbf{A}\mathbf{F}=0$. Matrix $\mathbf{\bar{D}}$ has
full column rank since it is the restriction of the full column rank matrix $\mathbf{D}$ to the null-space of $\mathbf{A}$.
\hfill $\blacksquare$

{\it Proof of Theorem \ref{lem:noGapRelax}:}
First notice that the feasibility set of LMI problem (\ref{primal}) is nonempty and bounded. Indeed, a triplet of zero measures is a trivial feasible point for~(\ref{rlp}) and hence $(0,0,0,\lambda)$ is feasible in~(\ref{rrlp}); consequently a concatenation of truncated moment sequences corresponding to the quadruplet of measures $(0,0,0,\lambda)$ is feasible in~(\ref{primal}) for each relaxation order $k$. Boundedness of the even components of each moment vector follows from the structure of the localizing matrices corresponding to the functions from Assumption~\ref{compact} and from the fact that the masses (zero-th moments) of the measures are bounded because of the constraint $\mu_0+\hat{\mu}_0 = \lambda$ and because $T < \infty$. Boundedness of the whole moment vectors then follows since the even moments appear on the diagonal of the positive semidefinite moment matrices.

To complete the proof, we follow \cite[Theorem 4]{trnovska} and show that boundedness of
the feasibility set of LMI problem (\ref{primal}) implies existence of an interior
point for LMI problem (\ref{dual}), and then from standard SDP
duality it follows readily that $p^*=d^*$ since $D$ has a full column rank;
see, e.g., \cite[Theorem 5]{trnovska} and references therein. 

Let $\mathbf{\bar{y}}$ denote a point in the feasibility set of LMI problem (\ref{primal}), i.e.
a vector satisfying
$\mathbf{A}\mathbf{\bar{y}}=\mathbf{b}$ and $\mathbf{e}+\mathbf{D}\mathbf{\bar{y}}
\in \mathbf{K}$. Suppose
that there is no interior point for LMI problem (\ref{dual}), i.e.
there is no $\mathbf{x}$ and $\mathbf{z}\in\mathrm{int}\:\mathbf{K}$ such
that $\mathbf{A}'\mathbf{x}+\mathbf{D}'\mathbf{z} = \mathbf{c}$. Then
from Lemma \ref{lem:alternative} there exists $\mathbf{y}\neq 0$ such that
$\mathbf{A}\mathbf{y}=0$, $\mathbf{D}\mathbf{y} \in \mathbf{K}$
and $\mathbf{c}'\mathbf{y} \leq 0$. It follows that there exists a ray
$\mathbf{\bar{y}}+t\mathbf{y}$, $t\geq 0$ of feasible points for LMI
problem (\ref{primal}), hence implying that the feasibility set is not bounded. \hfill $\blacksquare$

\section*{Acknowledgments}
The authors would like to thank Mathieu Claeys and Jean-Bernard Lasserre for helpful discussions, Masakazu Kojima for valuable insights into computational aspects and the anonymous referees for useful comments.

\end{document}